\documentclass[11pt]{amsart} 

\usepackage{verbatim, enumitem, geometry}
\usepackage{tikz}
\usetikzlibrary{arrows.meta}
\newtheorem{theorem}{Theorem}
\newtheorem{proposition}[theorem]{Proposition}
\newtheorem{lemma}[theorem]{Lemma}

\newtheorem{corollary}[theorem]{Corollary}

\theoremstyle{definition} 

\newtheorem{definition}[theorem]{Definition}

\newcommand{\spe}[1]{\mathbf{#1}}

\DeclareMathOperator{\Aut}{Aut}
\DeclareMathOperator{\Hilb}{Hilb}

\DeclareMathOperator{\lk}{link}
\DeclareMathOperator{\tr}{tr}
\DeclareMathOperator{\supp}{supp}

\DeclareMathOperator{\Fix}{Fix}

\DeclareMathOperator{\bdot}{\mathbf{\cdot}}

\DeclareMathOperator{\ps}{ps}

\DeclareMathOperator{\cyc}{cyc}
\DeclareMathOperator{\sgn}{sgn}
\DeclareMathOperator{\qsym}{QSym}

\allowdisplaybreaks

\title[Equivariant flag $f$-vectors]{On Equivariant flag $f$-vectors for balanced relative simplicial complexes}
\author{Jacob A. White}
\address{School of Mathematical and Statistical Sciences, University of Texas Rio Grande Valley, 1201 West University Drive, Edinburg, TX 78539 USA}

\begin{document}
\maketitle
\begin{abstract}
    We study the equivariant flag $f$-vector and equivariant flag $h$-vector of a balanced relative simplicial complex with respect to a group action. When the complex satisfies Serre's condition $(S_{\ell}),$ we show that the equivariant flag $h$-vector, the equivariant $h$-vector, and the equivariant $f$-vector satisfy several inequalities.
    
    We apply these results to the study of $P$-partitions of double posets, and weak colorings of mixed graphs.
\end{abstract}
\section{Introduction}
In Stanley's seminal paper \cite{stanley-action}, he studies group actions on graded posets. Given a group $\mathfrak{G}$ acting on a graded poset $P$, it also acts on the order complex $\Delta(P)$ of chains of $P$. Stanley proves several results about this action.

Stanley also proved several results about balanced Cohen-Macaulay complexes \cite{stanley-balanced}. A \emph{balanced simplicial complex} is a pure simplicial complex $\Delta$ on vertex set $V$ of dimension $d-1$, with a function $f:V \rightarrow [d]$ such that, for all $\sigma \in \Delta$ and $i \in [d]$, we have $|\sigma \cap f^{-1}(i)| \leq 1.$ Given $S \subseteq [d]$, we let $\Delta|_S = \{\sigma \in \Delta: f(v) \in S \mbox{ for every } v \in \sigma \}$. Stanley showed that if $\Delta$ is Cohen-Macaulay, then so is $\Delta|_S$ for all $S \subseteq [d]$. This result was recently generalized, by replacing the Cohen-Macaulay property with Serre's condition.

\begin{theorem}
Let $\Delta$ be a balanced simplicial complex of dimension $d-1$. Given $S \subseteq [d]$, if $\Delta$ satisfies Serre's condition $(S_{\ell})$, then so does $\Delta|_S$.
\end{theorem}

Our goal is to study balanced relative simplicial complexes, equipped with a group action by a finite group $\mathfrak{G}$. 
Let $\Phi$ be a balanced relative simplicial complex of dimension $d-1$. That is, $\Phi$ is a collection of subsets of $V$ such that, for  $\rho \subseteq \sigma \subseteq \tau \subseteq V$, if $\rho, \tau \in \Phi$, then $\sigma \in \Phi$. We also require that the maximal faces all have the same size $d$, and assume that there is a function $\kappa: V \rightarrow [d]$ such that, for every maximal face $\sigma \in \Phi$, $\kappa$ restricted to $\sigma$ is a bijection. Given $S \subseteq [d]$, we let $\Phi|_S = \{\sigma \in \Phi: \{\kappa(v): v \in \sigma \} \subseteq S \}$.
\begin{theorem}
Let $\Phi$ be a balanced relative simplicial complex on vertex set $V$ of dimension $d-1$ with coloring $\kappa: V \rightarrow [d].$ Suppose that $\Phi$ satisfies Serre's condition $(S_{\ell})$. Given $S \subseteq [d]$, we have $\Phi|_S$ also satisfies Serre's condition $(S_{\ell})$. 

If $\Phi$ is relatively Cohen-Macaulay, then so is $\Phi|_S$.
\label{thm:restriction}
\end{theorem}
Our primary goal is to apply this result to obtain information about flag enumeration for $\Phi$ with respect to a group action $\mathfrak{G}$.
Let $\mathfrak{G}$ be a group which acts on $V$. We say $\mathfrak{G}$ acts on $\Phi$ if for every $\mathfrak{g} \in \mathfrak{G}$, the following two conditions are satisfied:
\begin{enumerate}
    \item For every $\sigma \in \Phi$, we have $\{\mathfrak{g}v: v \in \sigma \} \in \Phi$.
    \item For every $v \in V$, we have $\kappa(\mathfrak{g}v) = \kappa(v)$.
\end{enumerate} 

Let $x_1, x_2, \ldots,$ be a sequence of commuting indeterminates. A power series $f$ in $x_1, \ldots$ is a \emph{quasisymmetric function} if $[x_{i_1}^{a_1}\cdots x_{i_k}^{a_k}]f = [x_1^{a_1}\cdots x_k^{a_k}] f$ for every sequence $(a_1, \ldots, a_k)$ and every increasing sequence $i_1 < i_2 < \cdots < i_k.$ The most common basis for the ring of quasisymmetric functions is the basis of monomial quasisymmetric functions. We define our quasisymmetric functions in terms of this basis.

For $\mathfrak{g} \in \mathfrak{G}$, we define the \emph{flag quasisymmetric class function} of $(\Phi, \mathfrak{G})$ to be 
\[ \Hilb(\Phi, \mathfrak{G}, \mathbf{x}; \mathfrak{g}) = \sum_{\sigma \in \Phi: \mathfrak{g}\sigma = \sigma}  M_{\kappa(\sigma), d+1}. \]

This is a quasisymmetric function. As we vary $\mathfrak{g} \in \mathfrak{G}$, we obtain a class function on $\mathfrak{G}$ whose values are quasisymmetric functions. Equivalently, we can write $\Hilb(\Phi, \mathfrak{G}, \mathbf{x}; \mathfrak{g}) = \sum_{S \subseteq [d]} f_S(\Phi, \mathfrak{G}; \mathfrak{g}) M_{S, d+1}$ for certain constants $f_S(\Phi, \mathfrak{G}; \mathfrak{g})$. Then $f_S(\Phi, \mathfrak{G})$ is also a character. We refer to the character $f_S(\Phi, \mathfrak{G})$ as the \emph{equivariant flag $f$-vector} of $(\Phi, \mathfrak{G})$. Our goal will be to prove inequalities about the equivariant flag $f$-vector. To that end, given two characters $\chi$ and $\rho$ of a group $\mathfrak{G}$, we write $\chi \geq_{\mathfrak{G}} \rho$ if $\chi - \rho$ is also a character. The equivariant flag $h$-vector is defined by 
\[h_S(\Phi, \mathfrak{G}) = \sum_{T: S \subseteq T} (-1)^{|T \setminus S|} f_S(\Phi, \mathfrak{G}). \]

\begin{theorem}
Let $\Phi$ be a balanced relative simplicial complex of dimension $d-1$ and let $\mathfrak{G}$ act on $\Phi$. Suppose $\Phi$ satisfies condition $(S_{\ell})$. Given $S \subseteq T \subseteq [d]$ with $|S| \leq \ell$, then 
\[\sum_{R: S \subseteq R \subseteq T} h_R(\Phi, \mathfrak{G}) \geq_{\mathfrak{G}} 0. \] 

In particular, if $\Phi$ is Cohen-Macaulay, then $h_T(\Phi, \mathfrak{G}) \geq_{\mathfrak{G}} 0$ for all $T \subseteq [d]$.
\label{thm:intro1}
\end{theorem}

There is also an equivariant $f$-vector, given by $f_{i-1}(\Phi, \mathfrak{G}) = \sum_{S \subseteq [d]: |S| = i} f_S(\Phi, \mathfrak{G})$. The equivariant $h$-vector is given by $h_i(\Phi, \mathfrak{G}) = \sum_{S \subseteq [d]: |S| = i} h_S(\Phi, \mathfrak{G})$. 
First, we state some new inequalities for the equivariant $h$-vector of a balanced relative simplicial complex which satisfies condition $(S_{\ell})$. 
\begin{theorem}
Let $\Phi$ be a balanced relative simplicial complex of dimension $d-1$. Let $\mathfrak{G}$ act on $\Phi$, and suppose that $\Phi$ satisfies $(S_{\ell})$. Then $h_i(\Phi, \mathfrak{G}) \geq_{\mathfrak{G}} 0$ for $i \leq \ell$. For $0 \leq i \leq \ell \leq j \leq d$, we have 
\[ \sum_{k=\ell}^{j} \binom{d-k}{j-k}\binom{k-\ell+i}{i}h_{k}(\Phi, \mathfrak{G}) \geq_{\mathfrak{G}} 0.\]
\label{thm:intro2}
\end{theorem}
For pure simplicial complexes which satisfy $(S_{\ell})$, and with $j=d$, these inequalities were previously shown in \cite{hvector1} and \cite{hvector2} for ordinary $h$-vectors.

We also give some inequalities for the equivariant $f$-vector.
\begin{theorem}
Let $\Phi$ be a balanced relative simplicial complex of dimension $d-1$ and let $\mathfrak{G}$ act on $\Phi$. Suppose $\Phi$ satisfies condition $(S_{\ell})$, and that $f_{i-1}(\Phi, \mathfrak{G}) = 0$ for $i < \ell$.
\begin{enumerate}
    \item For $i \geq \ell$, we have \[(d-i)f_{i-1}(\Phi, \mathfrak{G}) \leq_{\mathfrak{G}} (i-\ell+2)f_i(\Phi, \mathfrak{G}).\]
        \item For $\ell \leq i \leq \lfloor (d+\ell-1)/2 \rfloor$, we have $f_{i-1}(\Phi, \mathfrak{G}) \leq_{\mathfrak{G}} f_{i}(\Phi, \mathfrak{G})$. 
    \item For $\ell \leq i \leq \lfloor (d+\ell)/2 \rfloor$, we have $f_{i-1}(\Phi, \mathfrak{G}) \leq_{\mathfrak{G}} f_{d+\ell-i-1}(\Phi, \mathfrak{G})$. 
\end{enumerate}
\label{thm:intro3}
\end{theorem}
If we assume that the group action is trivial, and we define $a_i(\Phi) = f_{i+\ell-1}(\Phi, \mathfrak{G})$, then the latter two conditions become: 
\begin{enumerate}
    \item For $i \leq \lfloor (d-\ell-1)/2 \rfloor$, we have $a_i \leq a_{i+1}$.
    \item For $i \leq \lfloor (d-\ell)/2 \rfloor$, we have $a_i \leq a_{d-\ell-i}$.
\end{enumerate}
 In the literature, the sequence $(a_0, \ldots, a_{d-\ell})$ is called \emph{strongly flawless}. Thus, our Theorem states that if the first several entries of the sequence are $0$, then deleting them results in a strongly flawless sequence.

Our primary application is to double posets and mixed graphs.
A mixed graph is a generalization of a graph where we allow directed edges. A coloring of a mixed graph $G$ satisfies $f(u) \neq f(v)$ when $uv$ is an undirected edge, and $f(u) \leq f(v)$ when $(u,v)$ is a directed edge. Given an automorphism $\mathfrak{g}$ of $G$, we let \[\chi(G, \mathfrak{G}, \mathbf{x}; \mathfrak{g}) = \sum_{f: \mathfrak{g}f = f} \prod_{v \in V} x_{f(v)}.\]

Then $\chi(G, \mathfrak{G}, \mathbf{x})$ is a quasisymmetric class function, generalizing the chromatic symmetric function. We wished to determine when $[F_{S, n}] \chi(G, \mathfrak{G}, \mathbf{x})$ is an effective character. There is also a corresponding chromatic polynomial class function $\chi(G, \mathfrak{G}, x).$

A mixed graph is \emph{acyclic} if it does not contain a directed cycle. A \emph{mixed cycle} in $G$ is a cycle in the underlying undirected graph that has at least one directed edge. A \emph{near cycle} is a mixed cycle of length three that has exactly one undirected edge. We prove the following:
\begin{theorem}
Let $G$ be an acyclic mixed graph on $n$ vertices with no near cycles. Let $m(G)$ be the minimum number of undirected edges on a mixed cycle of $G$, with $m(G) = |G|$ if there are no mixed cycles. 

Let $\mathfrak{G}$ act on $G$. Then for $S \subseteq [n-1]$ with $|S| \leq m(G)$, we have \[[F_{S,n}]\chi(G, \mathfrak{G}, \mathbf{x}) \mbox{ is an effective character}. \]

We write $\chi(G, \mathfrak{G}, x) = \sum_{i=0}^n f_{i-2} \binom{x}{i} = \sum_{i=0}^n h_{i-1} \binom{x+n-i}{n}.$
\begin{enumerate}
    \item We have $h_i \geq_{\mathfrak{G}} 0$ for $i \leq \ell$.
    \item For $i \leq d-\ell$, we have $\sum_{j=\ell}^d \binom{i+j-\ell}{i} h_j \geq_{\mathfrak{G}} 0.$
\end{enumerate}

 Suppose that $m(G) \geq \chi(G)$.
\begin{enumerate}
        \item For $i \geq \chi(G)$, we have \[(d-i)f_{i-1} \leq_{\mathfrak{G}} (i-\chi(G)+2)f_i.\]
        \item For $\chi(G) \leq i \leq \lfloor (d+\chi(G)-1)/2 \rfloor$, we have $f_{i-1} \leq_{\mathfrak{G}} f_{i}$. 
    \item For $\chi(G) \leq i \leq \lfloor (d+\chi(G))/2 \rfloor$, we have $f_{i-1} \leq_{\mathfrak{G}} f_{d+\ell-i-1}$. 
\end{enumerate}
\label{thm:mixedgraphs}
\end{theorem}

We show that there is a balanced relative simplicial complex $\Phi(G)$ such that \[\Hilb(\Phi(G), \mathfrak{G}, \mathbf{x}) = \chi(G, \mathfrak{G}, \mathbf{x}).\] Under this identification, $f_{-2} = h_{-1} = 0$, and for $i \geq -1$ we have $f_i = f_i(\Phi(G))$, which explains our choice of notation in Theorem \ref{thm:mixedgraphs}. The complex $\Phi(G)$ is a generalization of Steingr\'imsson's coloring complex, and was previously introduced in \cite{white-1}. However, we give a self-contained introduction to this complex. We then show that this complex satisfies Serre's condition $(S_{m(G)}).$ Thus, we obtain a natural collection of balanced relative simplicial complexes which satisfy Serre's condition $(S_k)$ for any $k$.

A double poset $D$ is a set equipped with two partial orders. A double poset is \emph{tertispecial} if, whenever $m \prec_1 m'$, then $m$ and $m'$ are $\leq_2$-comparable. There is a natural quasisymmetric generating function $\Omega(D, \mathbf{x})$ associated to $D$, which was introduced by Grinberg \cite{grinberg} and enumerates $D$-partitions. 

A function $f: N \rightarrow \mathbb{N}$ is a $D$-partition if and only if it satisfies the following two properties:
\begin{enumerate}
    \item For $i \leq_1 j$ in $D$, we have $f(i) \leq f(j)$.
    \item For $i \leq_1 j$ and $j \leq_2 i$ in $D$, we have $f(i) < f(j)$.
\end{enumerate}

When a group $\mathfrak{G}$ acts on $D$, then it also acts on the set of $D$-partitions of a given weight, and we can define a quasisymmetric class function $\Omega(D, \mathfrak{G}, \mathbf{x})$. If we
write $\Omega(D, \mathfrak{G}, \mathbf{x}) = \sum_{S \subseteq {[n-1]}} \chi_{D, \mathfrak{G}, S} F_{S, n}$, then the $\chi_{D, \mathfrak{G}, S}$ are virtual characters. We proved the following:
\begin{theorem}[\cite{white-2}]
Let $D$ be a tertispecial double poset on $n$ elements, and let $\mathfrak{G}$ act on $D$. Given $S \subseteq [n-1]$, we have $[F_{S,n}] \Omega(D, \mathfrak{G}, \mathbf{x})$ is an effective character.
\label{thm:introposet2}
\end{theorem}
This paper gives an alternate proof.
 We show that, given a double poset $P$, there is a mixed graph $G$ such that $\Omega(P, \mathfrak{G}, \mathbf{x}) = \chi(G(P), \mathfrak{G}, \mathbf{x}).$ Thus we are able to determine results about double posets from corresponding results on mixed graphs.

The paper is organized as follows. First, we give an overview of quasisymmetric class functions in Section \ref{sec:prelim}. In Section \ref{sec:geometry}, we review the definition of balanced relative simplicial complexes, and discuss the equivariant flag $f$-vector. Then we discuss the Stanley-Reisner module and the equivariant $f$-vector. Then we connect groups actions on balanced relative simplicial complexes to characters on cohomology groups. In Section \ref{sec:serre}, we study Serre's condition, and then prove the various results mentioned in the introduction. In Section \ref{sec:mixedgraph}, we study mixed graphs and their chromatic quasisymmetric class functions. Finally, in Section \ref{sec:dbl} we study double posets.

\section{Preliminaries}
\label{sec:prelim}
Given a basis $B$ for a vector space $V$ over $\mathbb{C}$, and $\vec{\beta} \in B, \vec{v} \in V$, we let $[\vec{\beta}] \vec{v}$ denote the coefficient of $\vec{\beta}$ when we expand $\vec{v}$ in the basis $B$.

Let $\mathbf{x} = x_1, x_2, \ldots $ be a sequence of commuting indeterminates. Recall that an \emph{integer composition} $\alpha$ of a positive integer $n$ is a sequence $(\alpha_1, \ldots, \alpha_k)$ of positive integers such that $\alpha_1+\cdots + \alpha_k = n$. We write $\ell(\alpha) = k$, and $\alpha \models n$. Let $n \in \mathbb{N}$ and let $F(\mathbf{x}) \in \mathbb{C}[[\mathbf{x}]]$ be a homogeneous formal power series in $\mathbf{x}$, where the degree of every monomial in $F(\mathbf{x})$ is $n$. Then $F(\mathbf{x})$ is a \emph{quasisymmetric function} if it satisfies the following property:
for every $i_1 < i_2 < \cdots < i_k$, and every integer composition $\alpha \models n$ with $\ell(\alpha) = k$, we have $[\prod_{j=1}^k x_{i_j}^{\alpha_j}]F(\mathbf{x}) = [\prod_{j=1}^k x_j^{\alpha_j}]F(\mathbf{x})$. Often, we will define quasisymmetric functions that are generating functions over a collection of functions. Given a function $w: S \rightarrow \mathbb{N}$, we define 
\begin{equation}
    \mathbf{x}^w = \prod_{v \in S} x_{w(v)}.
    \label{eq:xf}
\end{equation} For example, the chromatic symmetric function of a graph $G$ is defined as $\chi(G, \mathbf{x}) = \sum\limits_{f:V \rightarrow \mathbb{N}} \mathbf{x}^f$ where the sum is over all proper colorings of $G$.

Given a subset $\{s_1, \ldots, s_k \} \subseteq [n-1]$, with $s_1 < s_2 < \cdots < s_k,$ we let \[M_{\{s_1, \ldots, s_k\}, n} = \sum\limits_{i_1 < \cdots < i_{k+1}} \prod_{j=1}^{k+1} x_{i_j}^{s_j - s_{j-1}}\] where we define $s_0 = 0$ and $s_{k+1} = n$. These are the \emph{monomial quasisymmetric functions}, which form a basis for the vector space of quasisymmetric functions of degree $n$. This basis is partially ordered, by saying $M_{S, n} \leq M_{T,n}$ if $S \subseteq T.$

The second basis we focus on is the basis of fundamental quasisymmetric functions, first introduced by Gessel \cite{gessel}.
The \emph{fundamental quasisymmetric functions} $F_{S,d}$ are defined by:
\[ F_{S, d} = \sum\limits_{T: S \subseteq T} M_{T, d}. \]

\subsection{Group actions and class functions}
\label{subsec:groupaction}

 Given a group action $\mathfrak{G}$ on a set $X$, we let $X / \mathfrak{G}$ denote the set of orbits. For $x \in X$, $\mathfrak{G}_x$ is the stabilizer subgroup, and $\mathfrak{G}(x)$ is the orbit of $x$. Finally, for $\mathfrak{g} \in \mathfrak{G}$, we let $\Fix_{\mathfrak{g}}(X) = \{x \in X: \mathfrak{g}x = x \}$.

We assume familiarity with the theory of complex representations of finite groups - see \cite{fulton-harris} for basic definitions. Recall that, given any group action of $\mathfrak{G}$ on a finite set $X$, there is a group action on $\mathbb{C}^X$ as well, which gives rise to a representation. The resulting representations are called \emph{permutation representations}. Given any $\mathfrak{G}$-set $X$, or $\mathfrak{G}$-module $V$, we let $\chi_{X, \mathfrak{G}}$ and $\chi_{V, \mathfrak{G}}$ denote the corresponding characters.
Let $R$ be a $\mathbb{C}$-algebra. Then an \emph{$R$-valued class function} is a function $\chi: \mathfrak{G} \rightarrow R$ such that, for every $\mathfrak{g}, \mathfrak{h} \in \mathfrak{G}$, and $\chi \in C(\mathfrak{G}, R)$, we have $\chi(\mathfrak{hg}\mathfrak{h}^{-1}) = \chi(\mathfrak{g}).$ Let $C(\mathfrak{G}, R)$ be the set of $R$-valued class functions from $\mathfrak{G}$ to $R$. For our paper, $R$ is usually $\qsym$ or $\mathbb{C}[x]$. We refer to $\chi \in C(\mathfrak{G}, \mathbb{C})$ as class functions when no confusion arises.

There is an orthonormal basis of $C(\mathfrak{G}, \mathbb{C})$ given by the characters of the irreducible representations of $\mathfrak{G}$. We refer to elements $\chi \in C(\mathfrak{G}, \mathbb{C})$ that are integer combinations of irreducible characters as \emph{virtual characters}, and elements that are nonnegative integer linear combinations as \emph{effective characters}. Let $E(\mathfrak{G}, \mathbb{C})$ be the set of effective characters. Finally, we say $\chi$ is a \emph{permutation character} if it is the character of a permutation representation. We partially order $E(\mathfrak{G}, \mathbb{C})$ by saying $\chi \leq_{\mathfrak{G}} \psi$ if $\psi - \chi$ is an effective character.

Let $\mathbf{B}$ be a basis for $R$. For $b \in \mathbf{B}$, $\mathfrak{g} \in \mathfrak{G}$, and $\chi \in C(\mathfrak{G}, R)$, let $\chi_b(\mathfrak{g}) = [b]\chi(\mathfrak{g})$. Then $\chi_b$ is a $\mathfrak{C}$-valued class function. Thus we can write $\chi = \sum\limits_{b \in \mathbf{B}} \chi_b b$. Conversely, given a family $\chi_b$ of $\mathfrak{C}$-valued class functions, one for each $b \in \mathbf{B}$, the function $\chi$ defined by $\chi(\mathfrak{g}) = \sum\limits_{b \in \mathbf{B}} \chi_b(\mathfrak{g}) b$ is an $R$-valued class function in $C(\mathfrak{G}, R)$.

Let $\chi$ be an $R$-valued class function. We say that $\chi$ is \emph{$\textbf{B}$-realizable} if $\chi_b$ is a permutation character for all $b$.

A \emph{quasisymmetric class function} is a $\qsym$-valued class function. Given a quasisymmetric class function $F(\mathfrak{G}, \mathbf{x})$, we write $F(\mathfrak{G}, \mathbf{x}) = \sum_{S \subseteq [n-1]} f_{S, \mathfrak{G}} M_{S, n}$, where $f_{S, \mathfrak{G}} \in C(\mathfrak{G}, \mathbb{C})$.

Finally, we define a function $\langle \cdot, \cdot \rangle: C(\mathfrak{G}, R) \times C(\mathfrak{G}, R) \rightarrow R$ by 
$\langle \chi, \psi \rangle = \frac{1}{|\mathfrak{G}|} \overline{\chi}(\mathfrak{g}) \psi(\mathfrak{g})$ where $\overline{x}$ is the complex conjugate. In the case where $R = \mathbb{C}$, this is the usual inner product on class functions.

\begin{proposition}
\label{prop:global} 
Let $\mathfrak{G}$ be a finite group, let $R$ be a $\mathbb{C}$-algebra with basis $\mathbf{B}$. Fix $\chi \in C(\mathfrak{G}, R)$.
\begin{enumerate}[label={(\arabic*)},itemindent=1em]
\item Given an irreducible character $\psi$, we have $\langle \chi, \psi \rangle = \sum\limits_{b, c \in B} \langle \chi_b, \psi_c \rangle b \cdot c$.
\item 
Let $\psi \in C(\mathfrak{G}, \mathbb{C})$. If $\chi$ is $B$-realizable, then for all $b  \in \textbf{B}$ we have $[b]\langle \psi, \chi \rangle \geq 0.$  \label{prop:charincreasing}
\end{enumerate}
\end{proposition}

\subsection{Polynomial class functions and Principal specialization}
Given a polynomial $p(x)$ of degree $d$, define the $f$-vector $(f_{-1}, \ldots, f_{d-1})$ via $p(x) = \sum_{i=0}^d f_{i-1} \binom{x}{i}$. We say that $p(x)$ is \emph{strongly flawless} if the following inequalities are satisfied:
\begin{enumerate}
    \item for $0 \leq i \leq \frac{d-1}{2}$, we have $f_{i-1} \leq f_{i}$.
    \item For $0 \leq i \leq \frac{d}{2}$, we have $f_{i-1} \leq f_{d-i-1}$.
\end{enumerate}
There is a lot of interest in log-concave and unimodal sequences in combinatorics. We consider strongly flawless sequences to be interesting, as strongly flawless unimodal sequences can be seen as a generalization of symmetric unimodal sequences. Examples of results with strongly flawless sequences include the work of Hibi \cite{hibi} and Juhnke-Kubitzke and Van Le \cite{kubitzke}.
We denote the entries of the $f$-vector as $f_i(p(x))$ when discussing multiple polynomials.

There is a generalization of $f$-vector for polynomial class functions, which we call the \emph{equivariant $f$-vector}. Given a group $\mathfrak{G}$, and a polynomial class function $p(\mathfrak{G}, x)$, we write $p(\mathfrak{G}, x) = \sum_{i=1}^d f_{i-1} \binom{x}{i}$, where the $f_i$ are characters.  
We say that $p(\mathfrak{G}, x)$ is \emph{effectively flawless} if we have the following system of inequalities:
\begin{enumerate}
    \item For $0 \leq i \leq \frac{d-1}{2}$, we have $f_{i-1} \leq_{\mathfrak{G}} f_{i}$.
    \item For $0 \leq i \leq \frac{d}{2}$, we have $f_{i-1} \leq_{\mathfrak{G}} f_{d-i-1}$.
\end{enumerate}

Given a quasisymmetric function $F(\mathbf{x})$ of degree $d$, there is an associated polynomial $\ps(F)(x)$ given by principal specialization. For $x \in \mathbb{N}$, we set \[x_i = \begin{cases} 1 & i \leq x \\ 0 & i > x
\end{cases} \]
The resulting sequence is a polynomial function in $x$ of degree $d$, which we denote by $\ps(F)(x)$. 
If we write $F(\mathbf{x}) = \sum\limits_{S \subseteq [d-1]} c_{S,d} M_{S,d}$, then $f_i(\ps(F(\mathbf{x}))) = \sum\limits_{S \subseteq [d-1]: |S| = i+1} c_{S, d}$.

Let $F(\mathfrak{G}, \mathbf{x})$ be a quasisymmetric class function of degree $d$, and $\mathfrak{g} \in \mathfrak{G}$. Define $\ps(F) \in C(\mathfrak{G}, \mathbb{C}[x])$ by $\ps(F)(x;\mathfrak{g}) = \ps(F(\mathbf{x}; \mathfrak{g}))(x).$ We refer to $\ps(F)$ as the principal specialization, which results in an polynomial class function.

The following results were obtained in \cite{white-2}.
\begin{proposition}
\label{prop:global2}
Let $F(\mathfrak{G}, \mathbf{x})$ be a quasisymmetric class function be of degree $d$, and $\mathfrak{g} \in \mathfrak{G}$.  
\begin{enumerate}[label={(\arabic*)},itemindent=1em]
    \item If we write $F(\mathfrak{G}, \mathbf{x}) = \sum\limits_{S \subseteq \models [d-1]} \chi_{S, d} M_{S, d}$, then $\ps(F) = \sum\limits_{i=0}^d \chi_{S, d} \binom{x}{|S|+1}$.
    \item If $F(\mathfrak{G}, \mathbf{x})$ is $M$-realizable and $M$-increasing, then $\ps(F)$ is effectively flawless. \label{prop:ftohorbit}
\end{enumerate}
\end{proposition}

\begin{proposition}
Let $(f_{-1}, \ldots, f_{d-1})$ be a sequence of characters for $\mathfrak{G}$. Suppose that there exists an integer $0 \leq r \leq d$ such that, for all $i$, we have \[ (d-i)f_{i-1} \leq_{\mathfrak{G}} (i-r)f_{i}. \]
Then the sequence $(f_{r-1}, \ldots, f_{d-1})$ is effectively flawless.
\label{prop:flawlesssequence}
\end{proposition}
\begin{proof}
We prove the result by induction on $r$. First, suppose that $r = 0$.

We see that $(i+1) \leq (d-i)$. Thus we have $(i+1)f_i \leq_{\mathfrak{G}} (d-i)f_{i-1} \leq_{\mathfrak{G}} (i+1)f_{i}$. Hence $(i+1)(f_{i}-f_{i-1})$ is an effective character. This implies that $(i+1)c_j$ is a non-negative integer for all $j$, which means $c_j$ is a non-negative rational number. Therefore $f_{i} - f_{i-1}$ is an effective character.

Now let $i \leq \lfloor d / 2 \rfloor$. We can generalize the inequality $(d-i)f_{i-1} \leq_{\mathfrak{G}} if_{i}$ to obtain $\binom{d-i}{j-i} f_{i-1} \leq_{\mathfrak{G}} \binom{j}{j-i}f_{j-1}$ whenever $i \leq j$. When we set $j = d-i$, then $\binom{d-i}{j-i} = \binom{j}{j-i}$, and thus we get $f_{i-1} \leq_{\mathfrak{G}} f_{d-i-1}$.

Now suppose that $r > 0$. Let $g_i = f_{i+1}$. Then we see that \[(d-1-i)g_{i-1} = (d-(i+1))f_{i} \leq_{\mathfrak{G}} (i+1-r)f_{i+1} = (i+(r-1))g_{i}.\] Thus, by the induction hypothesis, we have $(g_{r-2}, \ldots, g_{d-1})$ is effectively flawless. However, $(g_{r-2}, \ldots, g_{d-2}) = (f_{r-1}, \ldots, f_{d-1})$.
\end{proof}

\section{Balanced Relative Simplicial Complexes}
\label{sec:geometry}
Now we discuss balanced relative simplicial complexes, and their flag quasisymmetric class functions.

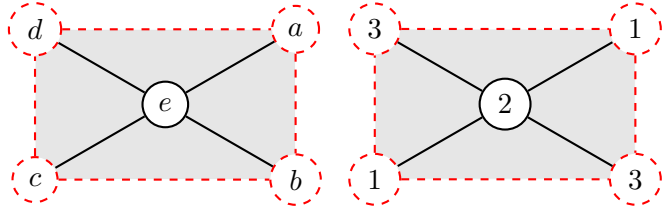
\begin{figure}
\begin{tabular}{cc}
\begin{tikzpicture}
\draw[color=white, fill=gray!20] (30:2cm) -- (330:2cm) -- (210:2cm) -- (150:2cm) -- cycle;
  \node[circle, draw=red, fill=white, dashed, thick] (b) at (150:2cm) {$d$};
  \node[circle, draw=red, fill=white, dashed, thick] (a) at (210:2cm) {$c$};
  \node[circle, draw=red, fill=white, dashed, thick] (e) at (330:2cm) {$b$};
  \node[circle, draw=red, fill=white,dashed, thick] (d) at (30:2cm) {$a$};
  \node[circle, draw=black, fill=white,thick] (or) at (0:0cm) {$e$};
  \draw[dashed, red, thick] (a) -- (b);
  \draw[dashed, red, thick] (d) -- (e);

  \draw[dashed, red, thick] (b) -- (d);
  \draw[dashed, red, thick] (a) -- (e);
  \draw[thick] (b) -- (or) -- (d);
  \draw[thick] (a) -- (or) -- (e);

\end{tikzpicture}
&
\begin{tikzpicture}
\draw[color=white, fill=gray!20] (30:2cm) -- (330:2cm) -- (210:2cm) -- (150:2cm) -- cycle;
  \node[circle, draw=red, fill=white, dashed, thick] (b) at (150:2cm) {$3$};
  \node[circle, draw=red, fill=white, dashed, thick] (a) at (210:2cm) {$1$};
  \node[circle, draw=red, fill=white, dashed, thick] (e) at (330:2cm) {$3$};
  \node[circle, draw=red, fill=white,dashed, thick] (d) at (30:2cm) {$1$};
  \node[circle, draw=black, fill=white,thick] (or) at (0:0cm) {$2$};
  \draw[dashed, red, thick] (a) -- (b);
  \draw[dashed, red, thick] (d) -- (e);

  \draw[dashed, red, thick] (b) -- (d);
  \draw[dashed, red, thick] (a) -- (e);
  \draw[thick] (b) -- (or) -- (d);
  \draw[thick] (a) -- (or) -- (e);

\end{tikzpicture}
\end{tabular}
    \caption{A coloring complex $\Phi$. Dashed lines correspond to faces that are not in $\Phi$.}
    \label{fig:coloringcomplex}
\end{figure}

\begin{definition}
A \emph{balanced relative simplicial complex} of dimension $d-1$ on a vertex set $V$ is a non-empty collection $\Phi$ of subsets of $V$, along with a function $\kappa: V \rightarrow [d]$ with the following properties:
\begin{enumerate}
    \item For every $\rho \subseteq \sigma \subseteq \tau$, if $\rho, \tau \in \Phi$, then $\sigma \in \Phi$.
    \item For every $\rho \in \Phi$, there exists $\sigma \in \Phi$ such that $\rho \subseteq \sigma$ and $|\sigma| = d$,
    \item For every $\rho \in \Phi$, we have $ \{\kappa(v): v \in \rho \}$ has size $|\rho|$.
\end{enumerate}
\end{definition}
The name comes from the fact that there exists simplicial complexes
$(\Delta, \Gamma)$ with $\Gamma \subseteq \Delta$, and $\Phi = \Delta \setminus \Gamma$. Given $\sigma \in \Phi$, we let $\kappa(\sigma) = \{\kappa(v): v \in \sigma \}$.

 Given $S \subseteq [d]$, we let $f_S(\Phi)$ denote the number of faces $\sigma$ such that $\kappa(\sigma) = S$. This is the \emph{flag $f$-vector} of $\Phi$. We encode the flag $f$-vector with a quasisymmetric function 
\[ \Hilb(\Phi, \mathbf{x})= \sum_{S \subseteq [d]} f_S(\Phi) M_{S, d+1}.\]
 This is the \emph{flag quasisymmetric function} associated to $\Phi$. It has degree $d+1$.

Let $V(\Phi)$ be the vertex set of $\Phi$. A bijection $\mathfrak{g}: V \rightarrow V$ is an \emph{automorphism} of $\Phi$ if it satisfies the following two properties:
\begin{enumerate}
    \item For every $v \in V$, we have $\kappa(\mathfrak{g}v) = \kappa(v)$.
    \item For every $\{v_1, \ldots, v_k \} \in \Phi$, we have $\{\mathfrak{g}(v_1), \ldots, \mathfrak{g}(v_k) \} \in \Phi$.
\end{enumerate}

Let $\mathfrak{G}$ be a group which acts on $V(\Phi)$. Then $\mathfrak{G}$ also acts on $2^{V(\Phi)}$ by $\mathfrak{g}(S) = \{\mathfrak{g}v: v \in S \}$. Suppose that, for each $\mathfrak{g} \in \mathfrak{G}$, the resulting action on $2^{V(\Phi)}$ is an automorphism of $\Phi$. Then we say $\mathfrak{G}$ acts on $\Phi$.
For $S \subseteq [d]$, let $\kappa^{-1}(S) = \{\sigma \in \Phi: \kappa(\sigma) = S \}$. Then $\mathfrak{G}$ acts on $\kappa^{-1}(S)$ as well. Given $\mathfrak{g} \in \mathfrak{G}$, we define 
\[ \Hilb(\Phi, \mathfrak{G}, \mathbf{x}; \mathfrak{g}) = \sum_{\sigma \in \Phi: \mathfrak{g}\sigma = \sigma}  M_{\kappa(\sigma), d+1}. \]
As we vary $\mathfrak{g}$, we obtain a quasisymmetric class function, which we call
 the \emph{flag quasisymmetric class function} of $(\Phi, \mathfrak{G})$, and denote $\Hilb(\Phi, \mathfrak{G}, \mathbf{x})$.

As an example, consider the balanced relative simplicial complex $\Phi$ with vertex set $\{a,b,c,d,e\}$ appearing on the left in Figure \ref{fig:coloringcomplex}. Here we note that \[\Phi = \{ \{e\}, \{a,e\}, \{b,e\}, \{c,e\}, \{d,e\}, \{a,b,e\}, \{b,c,e\}, \{c,d,e\}, \{a,d,e\} \}.\] The coloring $\kappa$ is given on the right in Figure \ref{fig:coloringcomplex}. Then $\Phi$ is a balanced relative simplicial complex. We see that $\Aut(\Phi)$ is isomorphic to $\mathbb{Z}/2\mathbb{Z} \times \mathbb{Z}/2\mathbb{Z}$. We let $\mathbb{Z}/2\mathbb{Z}$ act as the permutation $(ac)(bd)(e).$ Then \[\Hilb(\Phi, \mathbb{Z}/2\mathbb{Z}, \mathbf{x}) = M_{\{2\},4}+\rho(M_{\{1,2\},4}+M_{\{2,3\},4}+2M_{\{1,2,3\},4}).\]

We denote the coefficients of $[M_{S,d}]\Hilb(\Phi, \mathfrak{G}),$ by $f_S(\Phi, \mathfrak{G})$ and refer to them as the \emph{equivariant flag $f$-vector} of $\Phi$. These coefficients are permutation characters. Similarly, we can write \[\Hilb(\Phi, \mathfrak{G}, \mathbf{x}) = \sum_{S \subseteq [d]} h_S(\Phi, \mathfrak{G}) F_{S,d+1},\] where the $h_S(\Phi, \mathfrak{G})$ are virtual characters, which we call the \emph{equivariant flag $h$-vector}.

Since $\mathfrak{G}$ acts on $\Phi$, we can say two faces $\sigma$ and $\tau$ are $\mathfrak{G}$-equivalent if $\sigma = \mathfrak{g}\tau$ for $\mathfrak{g} \in \mathfrak{G}$. We see that $\kappa(\tau) = \kappa(\sigma)$. For $C \in \Phi/\mathfrak{G}$, we define $\kappa(C) = \kappa(\sigma)$ for any $\sigma \in C$. We see that $\kappa$ is well-defined on $\Phi/\mathfrak{G}$. We 
let \[\Hilb^O(\Phi, \mathfrak{G}, \mathbf{x}) = \sum_{C \in \Phi/\mathfrak{G}} M_{\kappa(C), d+1} \]
be the \emph{orbital flag quasisymmetric function}.

\begin{proposition}
Let $\Phi$ be a balanced relative simplicial complex of dimension $d-1$. Suppose that $\mathfrak{G}$ acts on $\Phi$. 
\begin{enumerate}
    \item Let $S \subseteq [d]$. Then $f_S(\Phi, \mathfrak{G}) = \chi_{\kappa^{-1}(S), \mathfrak{G}}$.
    \item For $S \subseteq [d]$, we have $f^O_S(\Phi, \mathfrak{G}) = |\kappa^{-1}(S)/\mathfrak{G}|.$
    \item We have \[\Hilb^O(\Phi, \mathfrak{G}, \mathbf{x}) = \frac{1}{|\mathfrak{G}|} \sum_{\mathfrak{g} \in \mathfrak{G}} \Hilb(\Phi, \mathfrak{G}, \mathbf{x}; \mathfrak{g}).\]
\end{enumerate}
\end{proposition}
\begin{proof}
For the first part, let $S \subseteq [d]$. Given $\mathfrak{g} \in \mathfrak{G}$, we see that $[M_{S, d+1}]\Hilb(\Phi, \mathfrak{G}, \mathbf{x}; \mathfrak{g}) = |\{\sigma \in \Phi: \kappa(\sigma) = S, \mathfrak{g}\sigma = \sigma \}| = |\Fix_{\mathfrak{g}}(\kappa^{-1}(S))|$. The first result follows. 

For the second result, we see that $[M_{S, d+1}]\Hilb^O(\Phi, \mathfrak{G}, \mathbf{x}) = |\{C \in \Phi/\mathfrak{G}: \kappa(C) = S \}| = |\kappa^{-1}(S)/\mathfrak{G}|$. 

For the third result, we have 
\begin{align*}
    \frac{1}{|\mathfrak{G}|} \sum_{\mathfrak{g} \in \mathfrak{G}} \Hilb(\Phi, \mathfrak{G}, \mathbf{x}; \mathfrak{g}) & = \frac{1}{|\mathfrak{G}|} \sum_{\mathfrak{g} \in \mathfrak{G}} \sum_{S \subseteq [d]} |\Fix_{\mathfrak{g}}(\kappa^{-1}(S))| M_{S, d+1} \\
    & = \sum_{S \subseteq [d]} \frac{1}{|\mathfrak{G}|} \sum_{\mathfrak{g} \in \mathfrak{G}}  |\Fix_{\mathfrak{g}}(\kappa^{-1}(S))| M_{S, d+1} \\
    & = \sum_{S \subseteq [d]}  |\kappa^{-1}(S)/\mathfrak{G}| M_{S, d+1} \\
    & = \Hilb^O(\Phi, \mathfrak{G}, \mathbf{x})
\end{align*}
where the first equality is our first identity, the third equality comes from Burnside's Lemma, and the last equality comes from our second identity.
\end{proof}

\subsection{Group action on homology}

Let $\Phi$ be a balanced simplicial complex of dimension $d-1$, with presentation $(\Delta, \Gamma)$. Suppose that $\mathfrak{G}$ acts on $\Phi$. We discuss various group actions on chain groups and homology groups. Here we let $\widetilde{C}_i(\Phi) = \widetilde{C}_i(\Delta, \Gamma)$, the relative chain group of the pair $(\Delta, \Gamma)$. Similarly, we let $\widetilde{H}_i(\Phi) = \widetilde{H}_i(\Delta, \Gamma)$.

Let $S \subseteq [d]$
We write $\Delta|_S = \{\sigma \subseteq V(\Phi): \sigma \subseteq \tau \mbox{ for some } \tau \in \Phi|_S \}$, and let $\Gamma|_S = \Phi|_S \setminus \Delta|_S$. We see that $\mathfrak{G}$ acts on both $\Delta|_S$ and $\Gamma|_S$ as simplicial complexes.
We let $\widetilde{C}_i(\Phi|_S) = \widetilde{C}_i(\Delta|_S) / \widetilde{C}_i(\Gamma|_S)$ denote the $i$th reduced chain group, which is generated by faces of dimension $i$. We see that $\partial_i: \widetilde{C}_i(\Delta|_S) \rightarrow \widetilde{C}_{i-1}(\Delta|_S)$ and $\partial_i: \widetilde{C}_i(\Gamma|_S) \rightarrow \widetilde{C}_{i-1}(\Gamma|_S)$ are both $\mathfrak{G}$-invariant. Thus $\partial_i: \widetilde{C}_i(\Phi|_S) \rightarrow \widetilde{C}_{i-1}(\Phi|_S)$ is also $\mathfrak{G}$-invariant. Hence $\mathfrak{G}$ acts on $\widetilde{H}_i(\Phi|_S)$. We let $\chi_{\widetilde{H}_{i-1}(\Phi|_S), \mathfrak{G}}$ denote the resulting character.

\begin{theorem}
Let $\Phi$ be a  balanced relative simplicial complex, and let $\mathfrak{G}$ act on $\Phi$. Then 
\[\Hilb(\Phi, \mathfrak{G}, \mathbf{x}) = \sum_{S \subseteq [d]} \left( \sum_{i=0}^{|S|} (-1)^{|S|-i} \chi_{\widetilde{H}_{i-1}(\Phi|_S), \mathfrak{G}} \right) F_{S, d}. \]

\label{thm:eulerchar2}
\end{theorem}
\begin{proof}
Let $S \subseteq [d]$, and let $\mathfrak{g} \in \mathfrak{G}$. Then
\begin{align*}
[F_{S, d}] \Hilb(\Phi, \mathfrak{G}, \mathbf{x}; \mathfrak{g}) & = 
\sum_{T \subseteq S} (-1)^{|S \setminus T|} [M_{T,d}] \Hilb(\Phi, \mathfrak{G}, \mathbf{x}; \mathfrak{g}) \\
& = \sum_{T \subseteq S} (-1)^{|S\setminus T|} |\{\sigma \in \kappa^{-1}(T): \mathfrak{g}\sigma = \sigma \}| \\
& = \sum_{i=0}^{|S|} (-1)^{|S|-i} \sum_{T \subseteq S: | T| = i} |\{\sigma \in F_T(\Phi|_S): \mathfrak{g} \sigma = \sigma \}| \\
& = \sum_{i=0}^{|S|} (-1)^{|S| - i} |\{\sigma \in \Phi|_S: |\kappa(\sigma)| = i, \mathfrak{g} \sigma = \sigma \}| \\
& = (-1)^{|S|}\sum_{i=0}^{|S|} (-1)^i \tr_{\mathfrak{g}}(\widetilde{C}_{i-1}(\Phi|_S)) \\ 
& = (-1)^{|S|}\sum_{i=0}^{|S|} (-1)^i \tr_{\mathfrak{g}}(\widetilde{H}_{i-1}(\Phi|_S)) \\
& = \sum_{i=0}^{|S|} (-1)^{|S|-i} \chi_{\widetilde{H}_{i-1}(\Phi|_S), \mathfrak{G}}( \mathfrak{g}).
\end{align*}
The first equality comes from the definition of fundamental quasisymmetric functions. The second equality uses the definition of the flag quasisymmetric class function. The third equality comes from choosing to sum over subsets by their size, and recognizing that if $\sigma \in \kappa^{-1}(T)$ and $T \subseteq S$, then $\sigma \in F_T(\Phi|_S)$. The fourth equality comes from simplifying the inner summation. The fifth and last equalities come from the definition of the related group actions. The sixth equality is the Hopf trace formula.

\end{proof}
\section{Complexes Satisfying Serre's Condition}
\label{sec:serre}

Given a balanced relative simplicial complex $\Phi$ and $v \in V(\Phi)$, we define $\lk_{\Phi}(v) = (\lk_{\Delta}(v), \lk_{\Gamma}(v))$. Similarly, given $S \subseteq V(\Phi)$, we let $\Phi \setminus S = (\Delta \setminus S, \Gamma \setminus S).$

A relative complex $\Phi$ with presentation $(\Delta, \Gamma)$ satisfies \emph{Serre's condition} $(S_{\ell})$ if \[\widetilde{H}_{i-1}(\lk_{\Delta}(\sigma), \lk_{\Gamma}(\sigma)) = 0\] for all $\sigma \in \Delta$ with $i \leq \min(\dim \lk_{\Phi}(\sigma), \ell-1).$ We say that $\Phi$ is \emph{relatively Cohen-Macaulay} if it satisfies $(S_{\dim \Phi})$.

Our first goal in this section is to prove Theorem \ref{thm:restriction}.
In the case $\Phi$ is a balanced simplicial complex, this result was shown by Holmes and Lyle \cite{holmes-lyle}. Our proof is a generalization of their proof.

We prove several lemmas first.
\begin{lemma}
Let $\Phi$ be a relative simplicial complex with presentation $(\Delta, \Gamma)$ which satisfies $(S_{\ell})$, and let $\sigma \in \Delta$. Then $\lk_{\Phi}(\sigma)$ also satisfies $(S_{\ell}).$
\label{lem:link}
\end{lemma}
\begin{proof}
 Let $\tau \in \lk_{\Phi}(\sigma)$. Then $\lk_{\lk_{\Phi}(\sigma)}(\tau) = \lk_{\Phi}(\sigma \cup \tau)$. Since $\Phi$ satisfies $(S_{\ell})$, we have $\widetilde{H}_{i-1}(\lk_{\Phi}(\sigma \cup \tau)) = 0$ for $i \leq \min(\dim(\lk_{\Phi}(\sigma \cup \tau)), \ell-1)$. Hence $\widetilde{H}_{i-1}(\lk_{\lk_{\Phi}(\sigma)}(\tau)) = 0$ for $i \leq \min(\dim(\lk_{\lk_{\Phi}(\sigma)}(\tau)), \ell - 1).$ Thus $\lk_{\Phi}(\sigma)$ satisfies $(S_{\ell}).$
\end{proof}

A subset $J \subseteq V(\Phi)$ is \emph{independent} if, for all $\sigma \in \Delta$, we have $|\sigma \cap J| \leq 1$. The set $J$ is \emph{excellent} if for every facet $\sigma \in \Delta$, we have $|\sigma \cap J| = 1$.
\begin{lemma}
Let $\Phi$ be a relative simplicial complex with presentation $(\Delta, \Gamma)$. Let $I$ be an independent set of $\Phi$ such that $\Delta \setminus I \neq \emptyset$. Suppose that $\Phi$ satisfies $(S_{\ell})$. Then $\widetilde{H}_{i-1}(\Phi \setminus I) = 0$ for $i \leq \min( \dim(\Phi \setminus I), \ell-1).$
\label{lem:independent}
\end{lemma}
\begin{proof}
We prove the result by induction on $\ell$, where the case $\ell = 1$ is trivial. So suppose $\ell \geq 2$. Since $\Phi$ satisfies $(S_{\ell})$, it also satisfies $(S_{\ell - 1})$, and thus by the induction hypothesis, we have $\widetilde{H}_{i-1}(\Phi \setminus I) = 0$ for $i \leq \min(\dim(\Phi \setminus I), \ell-2)$. Thus we only need to show that 
\begin{equation} \widetilde{H}_{\ell-2}(\Phi \setminus I) = 0 \mbox{ when }\ell -1 \leq \dim(\Phi \setminus I).\label{eq:homology} \end{equation} 

We prove Equation \eqref{eq:homology} by induction on $|I|.$
Suppose that $I = \{x\}$. Let $\Delta_1 = \{\sigma \in \Delta: \sigma \cup \{x\} \in \Delta \}$, and $\Gamma_1 = \{\sigma \in \Gamma: \sigma \cup \{x\} \in \Gamma \}$.
 Then $\Delta = \Delta_1 \cup \Delta \setminus \{x\}$, and $\Gamma = \Gamma_1 \cup \Gamma \setminus \{x\}$. Also, $\lk_{\Delta}(x) = \Delta_1 \cap \Delta \setminus \{x\}$ and $\lk_{\Gamma}(x) = \Gamma_1 \cap \Gamma \setminus \{x\}$. We let $\Phi_1 = (\Delta_1, \Gamma_1).$
 
 Consider the Mayer-Vietoris long exact sequence 
 \[ \widetilde{H}_{\ell-1}(\Phi_1) \oplus \widetilde{H}_{\ell-1}(\Phi \setminus \{x\}) \xrightarrow{i^{\ast}_{x}} \widetilde{H}_{\ell-1}(\Phi) \rightarrow \widetilde{H}_{\ell-1}(\lk_{\Phi}(x)) \rightarrow \widetilde{H}_{\ell-2}(\Phi_1) \oplus \widetilde{H}_{\ell-2}(\Phi \setminus \{x\}) \rightarrow \widetilde{H}_{\ell-2}(\Phi) \]

We claim that $\widetilde{H}_i(\Phi_1) = 0$ for all $i$. We see that $x$ is a cone point of $\Delta_1$. If $x \in \Gamma$, then $x$ is a cone point of $\Gamma_1$. Otherwise, $\Gamma_1 = \emptyset$. Thus, $\widetilde{H}_i(\Delta_1) = \widetilde{H}_i(\Gamma_1) = 0$ for all $i$. Applying the long exact sequence for a relative pair, it follows that $\widetilde{H}_i(\Phi_1) = 0$ for all $i$.

Since $\ell -1 \leq \dim(\Phi \setminus \{x\}) \leq \dim(\Phi)$, we see that $\ell - 2 \leq \dim(\lk_{\Phi}(x))$. Since $\Phi$ satisfies $(S_{\ell})$, it follows that $\widetilde{H}_{\ell-2}(\lk_{\Phi}(x)) = 0  = \widetilde{H}_{\ell-2}(\Phi).$ By exactness, we see that $\widetilde{H}_{\ell-2}(\Phi \setminus \{x\}) = 0$. Also, if we let $i_x$ denote the inclusion $\Phi\setminus \{x\}$ into $\Phi$, then $i^{\ast}_x$ is the induced map on homology. By exactness, $i^{\ast}_x$ is surjective.

Now suppose that $|I| > 2$, and let $x \in I$. Define $I' = I \setminus \{x\}$. Then $\Delta = \Delta \setminus I' \cup \Delta \setminus \{x\}$, and $\Delta \setminus I = \Delta \setminus I' \cap \Delta \setminus \{x\}$. Similarly, $\Gamma = \Gamma \setminus I' \cup \Gamma \setminus \{x\}$, and $\Gamma \setminus I = \Gamma \setminus I' \cap \Gamma \setminus \{x\}$.
 
 Consider the Mayer-Vietoris long exact sequence 
 \[ \widetilde{H}_{\ell-1}(\Phi \setminus I') \oplus \widetilde{H}_{\ell-1}(\Phi \setminus \{x\}) \xrightarrow{j^{\ast}_{I'} - i^{\ast}_{x}} \widetilde{H}_{\ell-1}(\Phi) \rightarrow \widetilde{H}_{\ell-2}(\Phi \setminus I) \rightarrow \widetilde{H}_{\ell-2}(\Phi \setminus I') \oplus \widetilde{H}_{\ell-2}(\Phi \setminus \{x\}). \]
 Here $j^{\ast}_{I'}$ is induced by the inclusion map $j_{I'}: \Phi \setminus I' \to \Phi$.
 By induction, we see that $\widetilde{H}_{\ell-2}(\Phi \setminus I') = 0 = \widetilde{H}_{\ell-2}(\Phi \setminus \{x\})$. Also, $j_{I'}$ is the inclusion map of $\Phi \setminus I'$ into $\Phi$. Since $i_x^{\ast}$ is surjective, it follows from exactness that $\widetilde{H}_{\ell-2}(\Phi \setminus I) = 0$.

\end{proof}

\begin{lemma}
Let $\Phi$ be a relative simplicial complex with presentation $(\Delta, \Gamma)$. Let $J$ be an excellent set of $\Phi$. Suppose that $\Phi$ satisfies $(S_{\ell}).$ Then $\Phi \setminus J$ satisfies $(S_{\ell}).$
\label{lem:excellent}
\end{lemma}
\begin{proof}
We prove the result by induction on $\dim \Phi$. Let $\sigma \in \Delta \setminus J$ with $|\sigma| > 0$. By Lemma \ref{lem:link}, we know that $\lk_{\Phi}(\sigma)$ satisfies $(S_{\ell})$. Define $J' = V(\lk_{\Phi}(\sigma)) \cap J$. We claim that $J'$ is excellent. If $J'$ is not independent as a subset of $\lk_{\Phi}(\sigma)$, then $J'$ is not independent in $\Phi$, and thus $J$ is not independent. Therefore, $J'$ is an independent set. Let $\tau$ be a facet of $\lk_{\Phi}(\sigma)$. Then $\tau \cup \sigma$ is a facet of $\Phi$. Hence $(\tau \cup \sigma) \cap J = \{x\}$ for some unique $x \in J$. Since $\sigma \in \Delta \setminus J$, we see that $x \in \tau$, and thus $|\tau \cap J'| \leq 1$. If $|\tau \cap J'| \geq 2$, then $|(\sigma \cup \tau) \cap J| \geq 2$, a contradiction. Therefore, $J'$ is excellent.

Since $\dim(\lk_{\Phi}(\sigma)) < \dim(\Phi)$, by induction $\lk_{\Phi}(\sigma) \setminus J'$ satisfies $(S_{\ell}).$ 
Observe that $\lk_{\Phi \setminus J}(\sigma) = \lk_{\Phi}(\sigma) \setminus J'$. Thus $\widetilde{H}_{i-1}(\lk_{\Phi \setminus J}(\sigma)) = 0$ for $i \leq \min(\dim(\lk_{\Phi \setminus J}(\sigma)), \ell-1).$

It remains to show that $\widetilde{H}_{i-1}(\Phi \setminus J) = 0$ for $i \leq \min(\dim(\Phi \setminus J), \ell-1)$, which follows from Lemma \ref{lem:independent}.
\end{proof}
\begin{proof}[Proof of Theorem \ref{thm:restriction}]
Let $\Phi$ be a balanced relative simplicial complex of dimension $d-1$ which satisfies $(S_{\ell})$. Let $S \subseteq [d]$. We prove the result by induction on $|[d] \setminus S|.$ The base case where $S = [d]$ is true by assumption.

Let $x \in S$, and let $S' = S \setminus \{x\}$, and let $J = \kappa^{-1}(x)$. By induction, $\Phi|_{S'}$ satisfies $(S_{\ell})$. We see that $J$ is excellent in $\Phi|_{S'}$, and that $\Phi|_S = \Phi|_{S'} \setminus J$. Thus the result follows from Lemma \ref{lem:excellent}.
\end{proof}

We state a few properties about $(S_{\ell})$ that we need.
\begin{proposition}
Let $\Phi$ be a balanced simplicial complex of dimension $d-1$ with presentation $(\Delta, \Gamma)$. Suppose that $\Delta$ satisfies $(S_{\ell})$ and $\Gamma$ satisfies $(S_{\ell - 1})$. If $\Gamma$ has dimension $d-2$, or is void, then $\Phi$ satisfies $(S_{\ell})$. 
\label{prop:relative}
\end{proposition}
\begin{proof}
Let $\sigma \in \Delta$. Consider the long exact sequence \[ \widetilde{H}_{i-1}(\lk_{\Delta}(\sigma)) \rightarrow \widetilde{H}_{i-1}(\lk_{\Phi}(\sigma)) \rightarrow \widetilde{H}_{i-2}(\lk_{\Gamma}(\sigma)) \]
We see that, when $i \leq \min(\dim(\lk_{\Phi}(\sigma)), \ell-1)$, then $\widetilde{H}_i(\lk_{\Delta}(\sigma)) = 0$. If $\lk_{\Gamma}(\Phi) = \emptyset$, then $\widetilde{H}_{i-1}(\lk_{\Gamma}(\sigma)) = 0$. Otherwise, $\dim(\lk_{\Gamma}(\sigma)) = \dim(\lk_{\Delta}(\sigma))-1$. Then since $\Gamma$ satisfies $(S_{\ell-1})$, we have $\widetilde{H}_{i-2}(\lk_{\Gamma}(\sigma)) = 0$. By exactness, $\widetilde{H}_{i-1}(\lk_{\Phi}(\sigma)) = 0.$
\end{proof}

\begin{proposition}
Let $\Delta_1$ and $\Delta_2$ be two $d$-dimensional simplicial complexes such that $\Delta_1 \cap \Delta_2$ is $(d-1)$-dimensional. If $\Delta_1$ and $\Delta_2$ both satisfy $(S_{\ell})$, and $\Delta_1 \cap \Delta_2$ satisfies $(S_{\ell - 1})$, then $\Delta_1 \cup \Delta_2$ satisfies $(S_{\ell})$.
\label{prop:constructible}
\end{proposition}
\begin{proof}
Let $\sigma \in \Delta_1 \cup \Delta_2$, and $i \leq \min( d-|\sigma|, \ell-1)$. If $\sigma \in \Delta_1 \setminus \Delta_2$, then $\lk_{\Delta_1 \cup \Delta_2}(\sigma) = \lk_{\Delta_1}(\sigma)$, and $\widetilde{H}_{i-1}(\lk_{\Delta_1 \cup \Delta_2}(\sigma)) = 0$, since $\Delta_1$ satisfies $(S_{\ell})$. A similar result holds if $\sigma \in \Delta_2 \setminus \Delta_1$.

So we may assume $\sigma \in \Delta_1 \cap \Delta_2.$ Let $\Gamma_i = \lk_{\Delta_i}(\sigma)$, and let $\Gamma = \lk_{\Delta_1 \cup \Delta_2}(\sigma)$. Then $\Gamma = \Gamma_1 \cup \Gamma_2$, and $\Gamma_1 \cap \Gamma_2 = \lk_{\Delta_1 \cap \Delta_2}(\sigma)$. 
We consider the Mayer-Vietoris sequence
\[\widetilde{H}_{i-1}(\Gamma_1) \oplus \widetilde{H}_{i-1}(\Gamma_2) \rightarrow \widetilde{H}_{i-1}(\Gamma) \rightarrow \widetilde{H}_{i-2}(\Gamma_1 \cap \Gamma_2). \]
If $0 \leq i \leq \min(d-|\sigma|, \ell-1)$, then since $\Delta_j$ satisfies $(S_{\ell})$, we see that $\widetilde{H}_{i-1}(\Gamma_j) = 0$. Similarly, since $\Delta_1 \cap \Delta_2$ satisfies $(S_{\ell - 1})$, we see that $\widetilde{H}_{i-2}(\Gamma_1 \cap \Gamma_2) = 0$. Thus, by exactness, $\widetilde{H}_{i-1}(\Gamma) = 0$, and $\Delta_1 \cup \Delta_2$ satisfies $(S_{\ell}).$
\end{proof}
\subsection{Inequalities}
Now we discuss inequalities that hold for various equivariant flag $f$- and $h$-vectors for balanced relative simplicial complexes which satisfy $(S_{\ell})$. First, we introduce a generalization $h_{S,T}(\Phi, \mathfrak{G})$ of the flag $h$-vector and prove some identities about this generalization. A lot of inequalities come from expressing some summation of flag $h$-vectors in terms of $h_{S,T}(\Phi, \mathfrak{G}).$
\begin{theorem}

Let $\Phi$ be a balanced relative simplicial complex of dimension $d-1$. Let $\mathfrak{G}$ act on $\Phi$.
Given $S \subseteq T \subseteq [d]$, define \[h_{S,T}(\Phi, \mathfrak{G}) = \sum_{R: S \subseteq R \subseteq T}  h_R(\Phi|_T, \mathfrak{G}).\] 

Let $\mathcal{T}_{T \setminus S}(\Phi)$ be a transversal for $\mathfrak{G}$ acting on $F_{T \setminus S}(\Phi).$
Then \begin{align*} h_{S,T}(\Phi, \mathfrak{G}) &= \sum_{Q: T \setminus S \subseteq Q \subseteq T} (-1)^{|T \setminus Q|} f_Q(\Phi|_T, \mathfrak{G}) \\ &= \sum_{\tau \in \mathcal{T}_{T \setminus S}(\Phi)} h_{S} ((\lk_{\Phi}(\tau))|_{S}, \mathfrak{G}_{\tau}) \uparrow_{\mathfrak{G}_{\tau}}^{\mathfrak{G}}.\end{align*}

\label{lem:fundamental}
\end{theorem}
\begin{proof}

First, we show that \[ \sum_{Q: T \setminus S \subseteq Q \subseteq T} (-1)^{|T \setminus Q|} f_Q(\Phi|_T, \mathfrak{G}) = \sum_{R:  S \subseteq R \subseteq T} h_R(\Phi|_T, \mathfrak{G}). \]

We write 
\begin{align*}
    \sum_{Q: T \setminus S \subseteq Q \subseteq T} (-1)^{|T \setminus Q|} f_Q(\Phi|_T, \mathfrak{G}) & = \sum_{Q: T \setminus S \subseteq Q \subseteq T} (-1)^{|T \setminus Q|} \sum_{R \subseteq Q} h_R(\Phi|_T, \mathfrak{G}) \\
    & = \sum_{R \subseteq T} \left( \sum_{Q: R \cup (T \setminus S) \subseteq Q \subseteq T} (-1)^{|T \setminus Q|} \right) h_R(\Phi|_T, \mathfrak{G}).
\end{align*}
We see that the coefficient of $h_R(\Phi|_T, \mathfrak{G})$ is $0$ unless $R \cup (T \setminus S) = T$, in which case it is $1$. Thus
\begin{align*}
    \sum_{Q: (T \setminus S) \subseteq Q \subseteq T} (-1)^{|T \setminus Q|} f_Q(\Phi|_T, \mathfrak{G}) & = \sum_{R: R \cup (T \setminus S) = T} h_R(\Phi, \mathfrak{G}) \\
    & = \sum_{R: S \subseteq R \subseteq T} h_R(\Phi, \mathfrak{G}).
\end{align*}

 Let $\mathfrak{g} \in \mathfrak{G}$. 
 We prove that \begin{equation} \sum_{Q: (T \setminus S) \subseteq Q \subseteq T} (-1)^{|T \setminus Q|} f_Q(\Phi|_T, \mathfrak{G}; \mathfrak{g}) = \sum_{\tau \in \Fix_{\mathfrak{g}}(\Phi|_{T \setminus S})} h_{S}((\lk_{\Phi}(\tau))|_{S}, \mathfrak{G}_{\tau}, \mathfrak{g}). \label{eq:evaluated} \end{equation}
 First, observe that
 \begin{align*}
     \sum_{Q: (T \setminus S) \subseteq Q \subseteq T} (-1)^{|T \setminus Q|} f_Q(\Phi|_T, \mathfrak{G}; \mathfrak{g})
     & = \sum_{\sigma \in \Fix_{\mathfrak{g}}(\Phi|_T): (T \setminus S) \subseteq \kappa(\sigma)} (-1)^{|T \setminus \kappa(\sigma)|} \\
     & = \sum_{\tau \in \Fix_{\mathfrak{g}}(\Phi|_{T \setminus S})} \sum_{\sigma \in \Fix_{\mathfrak{g}}(\Phi|_T): \tau \subseteq \sigma} (-1)^{|T \setminus\kappa(\sigma)|} \\
     & = \sum_{\tau \in \Fix_{\mathfrak{g}}(\Phi|_{T \setminus S})} \sum_{\sigma \in \Fix_{\mathfrak{g}}(\lk_{\Phi|_T}(\tau))} (-1)^{|S \setminus  \kappa(\sigma))|}.
 \end{align*}
 
 For a fixed $\tau$, the inner summation becomes 
 \begin{align*}
     \sum_{\sigma \in \Fix_{\mathfrak{g}}(\lk_{\Phi|_T}(\tau))} (-1)^{|S \setminus \kappa(\sigma)|} &= \sum_{R: R \subseteq S} (-1)^{|(S \setminus R)|} f_R(\lk_{\Phi}(\tau)|_{ S}, \mathfrak{G}_{\tau}; \mathfrak{g}) \\ &= h_{ S}(\lk_{\Phi}(\tau)|_{S}, \mathfrak{G}_{\tau}; \mathfrak{g})
 \end{align*}
 Hence we have shown Equation \eqref{eq:evaluated}.

Let $\mathcal{T}_{T \setminus S}(\Phi)$ be a traversal for $\mathfrak{G}$ acting on $\kappa^{-1}(T \setminus S)$. Fix $\tau \in \mathcal{T}_{T \setminus S}(\Phi)$. Then 
\begin{align*} h_{S} ((\lk_{\Phi}(\tau))|_{S}, \mathfrak{G}_{\tau}) \uparrow_{\mathfrak{G}_{\tau}}^{\mathfrak{G}}(\mathfrak{g}) &= \frac{1}{|\mathfrak{G}_{\tau}|} \sum_{\substack{\mathfrak{h} \in \mathfrak{G} \\ \mathfrak{h}^{-1}\mathfrak{gh} \in \mathfrak{G}_{\tau}}} h_{S} ((\lk_{\Phi}(\tau))|_{S}, \mathfrak{G}_{\tau};\mathfrak{h}^{-1}\mathfrak{gh}) \\ 
&= \frac{1}{|\mathfrak{G}_{\tau}|} \sum_{\substack{\mathfrak{h} \in \mathfrak{G} \\ \mathfrak{g} \in \mathfrak{G}_{\mathfrak{h}\tau}}} h_{S} ((\lk_{\Phi}(\mathfrak{h}\tau))|_{S}, \mathfrak{G}_{\mathfrak{h}\tau};\mathfrak{g}) \\
&= \frac{1}{|\mathfrak{G}_{\tau}|} \sum_{\substack{\rho \simeq \tau \\ \mathfrak{g} \in \mathfrak{G}_{\rho}}} \sum_{\substack{\mathfrak{h} \in \mathfrak{G} \\ \mathfrak{h}\tau = \rho}} h_{S} ((\lk_{\Phi}(\rho))|_{S}, \mathfrak{G}_{\rho};\mathfrak{g}) \\
& =  \sum_{\substack{\rho \simeq \tau \\ \mathfrak{g} \in \mathfrak{G}_{\rho}}}  h_{S} ((\lk_{\Phi}(\rho))|_{S}, \mathfrak{G}_{\rho};\mathfrak{g}) \left( \sum_{\substack{\mathfrak{h} \in \mathfrak{G} \\ \mathfrak{h}\tau = \rho}} \frac{1}{|\mathfrak{G}_{\tau}|} \right) \\
& = \sum_{\substack{\rho \simeq \tau \\ \mathfrak{g} \in \mathfrak{G}_{\rho}}}  h_{S} ((\lk_{\Phi}(\rho))|_{S}, \mathfrak{G}_{\rho};\mathfrak{g})
\end{align*}
 where the first equality is a formula for induced class functions. The second equality comes from applying conjugation. The third equality is the result of replacing one summation with two summmations, the outer summation being over faces $\rho$ that are equivalent to $\tau$, and then over group elements $\mathfrak{h}$ such that $\mathfrak{h}\tau = \rho$. 
Thus, summing over all $\tau \in \mathcal{T}_{T \setminus S}(\Phi)$, we obtain 
\begin{align*}
    \sum_{\tau \in \mathcal{T}_{T \setminus S}(\Phi)} h_{S} ((\lk_{\Phi}(\tau))|_{S}, \mathfrak{G}_{\tau}) \uparrow_{\mathfrak{G}_{\tau}}^{\mathfrak{G}}(\mathfrak{g}) &= 
    \sum_{\tau \in \mathcal{T}_{T \setminus S}(\Phi)} \sum_{\substack{\rho \simeq \tau \\ \mathfrak{g} \in \mathfrak{G}_{\rho}}}  h_{S} ((\lk_{\Phi}(\rho))|_{S}, \mathfrak{G}_{\rho};\mathfrak{g}) \\ 
    & = \sum_{\substack{\tau \in \Fix_{\mathfrak{g}}(\Phi) \\ \kappa(\tau) = T \setminus S}} h_{S} ((\lk_{\Phi}(\tau))|_{S}, \mathfrak{G}_{\tau};\mathfrak{g}) \\
    & = \sum_{Q: (T \setminus S) \subseteq Q \subseteq T} (-1)^{|T \setminus Q|} f_Q(\Phi|_T, \mathfrak{G}; \mathfrak{g}).
\end{align*}
\end{proof}

\begin{theorem}
Let $\Phi$ be a balanced relative simplicial complex of dimension $d$-1. Suppose that $\mathfrak{G}$ acts on $\Phi$, and that $\Phi$ satisfies $(S_{\ell}).$

Give $S \subseteq T \subseteq [d]$ with $|S| \leq \ell$, let $\mathcal{T}$ be a transversal for $\mathfrak{G}$ acting on $\kappa^{-1}(T \setminus S)$. Then we have 
\begin{align*} h_{S,T}(\Phi, \mathfrak{G}) = \sum_{\tau \in \mathcal{T}_{T \setminus S}(\Phi)} \chi_{\widetilde{H}_{|S|-2}(\lk_{\Phi}(\tau)|_S), \mathfrak{G}_{\tau}} \uparrow_{\mathfrak{G}_{\tau}}^{\mathfrak{G}}.\end{align*}

In particular, $h_{S,T}(\Phi, \mathfrak{G}) \geq_{\mathfrak{G}} 0.$
\label{thm:interpretation}
\end{theorem}
Observe that this result implies Theorem \ref{thm:intro1}.
\begin{proof}
Now suppose that $\Phi$ satisfies $(S_{\ell})$ and $|S| \leq \ell$. Let $\tau \in \mathcal{T}_{T \setminus S}(\Phi)$. By Theorem \ref{thm:eulerchar2},we have \begin{equation} h_S((\lk_{\Phi}(\tau))|_S, \mathfrak{G}_{\tau}) = \sum_{i=0}^{|S|} (-1)^{|S|-i} \chi_{\widetilde{H}_{i-1}(\lk_{\Phi}(\tau)|_S), \mathfrak{G}_{\tau}}.  \label{eq:alternating} \end{equation}
Since $\Phi$ satisfies $(S_{\ell})$, we know $\Phi|_T$ satisfies $(S_{\ell})$ as well, and hence $\widetilde{H}_{i-1}(\lk_{\Phi|_T}(\tau)) = 0$ for $i \leq \min(\dim(\lk_{\Phi|_T}(\tau)), \ell-1)$. Since $\dim(\lk_{\Phi|_T}(\tau)) = |S|-1$, and $|S| \leq \ell$, we see that the homology is concentrated in the top dimension. Moreover, $\lk_{\Phi|_T}(\tau) = \lk_{\Phi}(\tau)|_S$. Thus all the terms on the right hand side of Equation \eqref{eq:alternating} are $0$, except when $i = |S|$. Hence \[h_S((\lk_{\Phi}(\tau))|_S, \mathfrak{G}_{\tau}) = \chi_{\widetilde{H}_{|S|-2}(\lk_{\Phi}(\tau)|_S), \mathfrak{G}_{\tau}}. \]

Hence $h_S((\lk_{\Phi}(\tau))|_S, \mathfrak{G}_{\tau})$ is an effective character, and thus $h_{S,T}(\Phi, \mathfrak{G})$ is also an effective character.
\end{proof}

We now prove a Lemma that will be used to prove Theorems \ref{thm:intro2} and \ref{thm:intro3}. 



\begin{lemma}
Let $\Phi$ be a balanced relative simplicial complex of dimension $d-1$. Let $\mathfrak{G}$ act on $\Phi$. Suppose that $\Phi$ satisfies $(S_{\ell})$, and let $S \subseteq [d]$. Let $i \leq \ell$. Then
\[\sum_{T \subseteq S: |T| \geq \ell} \binom{|T| - (\ell-i)}{i} h_T(\Phi, \mathfrak{G}) \geq_{\mathfrak{G}} 0. \]
\label{lem:fundamental2}
\end{lemma}
\begin{proof}

For a set $S \subseteq [d]$, and $i < d$, we let $m_i(S)$ be the $(\ell-i)$th smallest element, and let \[M_i(S) = \{x \in S: x \leq m_i(S) \}.\] So if we write $S = \{s_1, \ldots, s_k \}$ with $s_1 < s_2 < \cdots < s_k$, then $m_i(S) = s_{\ell-i}$ and $M_i(S) = \{s_1, \ldots, s_{\ell-i} \}$.

Given $i \leq \ell$, and $|T| \geq \ell$, we see that $\binom{|T|-(\ell-i)}{i}$ counts the number of subsets $R \subseteq T$ with $|R| = \ell$ and $M_i(R) = M_i(T)$.

We see that \begin{align*} \sum_{T \subseteq S: |T| \geq \ell} \binom{|T| - (\ell-i)}{i} h_T(\Phi, \mathfrak{G}) & = \sum_{T \subseteq S: |T| \geq \ell} \sum_{\substack{R \subseteq T: |R| = \ell \\ M_i(R) = M_i(T)}} h_T(\Phi, \mathfrak{G}) \\
& = \sum_{R \subseteq [d]: |R| = \ell } \sum_{\substack{T: R \subseteq T \subseteq S\\ M_i(R) = M_i(T)}}  h_T(\Phi, \mathfrak{G}) \\ 
& = \sum_{R \subseteq [d]: |R| = \ell } \sum_{\substack{T: R \subseteq T \\ T \setminus M_{ i}(R) \subseteq S \setminus [m_i(R)]}}  h_T(\Phi, \mathfrak{G}).
\end{align*}
The first equality comes from our interpretation of the binomial coefficient, while the second equality comes from rearranging the order of summation. The third equality comes from recognizing that $R \subseteq T \subseteq S$ with $M_i(R) = M_i(T)$ holds if and only if $T \setminus M_i(R) \subseteq S \setminus [m_i(R)].$ For each $R$, the inner sum becomes \[ \sum_{T: R \subseteq T \subseteq M_i(R) \cup (S \setminus [m_i(R)])} h_T(\Phi, \mathfrak{G}) = h_{R, M_i(R) \cup (S \setminus [m_i(R)])}(\Phi, \mathfrak{G}) \]
which is an effective character by Lemma \ref{lem:fundamental}, since $|R| = \ell$. Thus we have a sum of effective characters, which is effective. Hence \[\sum_{T \subseteq S: |T| \geq \ell} \binom{|T| - (\ell-i)}{i} h_T(\Phi, \mathfrak{G}) \geq_{\mathfrak{G}} 0 .\]\end{proof}

\begin{proof}[Proof of Theorem \ref{thm:intro3}]
Let $\Phi$, $d$, $\mathfrak{G}$ be as in the statement of the Theorem. Suppose $f_{i-1}(\Phi, \mathfrak{G}) = 0$ for all $i < \ell$, and that $\Phi$ satisfies $(S_{\ell})$. 

Let $i \geq \ell$. Then we have
\begin{align*}
    & (i-\ell+2)f_i(\Phi, \mathfrak{G}) & - & (d-i) f_{i-1}(\Phi, \mathfrak{G}) &= \\
    & \sum_{\substack{S \subseteq [d] \\ |S| = i+1}} (|S|-(\ell-1))f_S(\Phi, \mathfrak{G}) & - & \sum_{R \subseteq [d]: |R| = i} (d-|R|) f_R(\Phi, \mathfrak{G}) &= \\
    & \sum_{\substack{S \subseteq [d] \\ |S| = i+1}} (|S|-(\ell-1))f_S(\Phi, \mathfrak{G}) & - & \sum_{R \subseteq [d]: |R| = i} \sum_{x \in [d]\setminus R} f_R(\Phi, \mathfrak{G}) &= \\
    & \sum_{\substack{S \subseteq [d] \\ |S| = i+1}} (|S|-(\ell-1))f_S(\Phi, \mathfrak{G}) & - & \sum_{S \subseteq [d]: |S| = i+1} \sum_{x \in S} f_{S \setminus \{x\}}(\Phi, \mathfrak{G}) &= \\
    & \sum_{\substack{S \subseteq [d] \\ |S| = i+1}} \sum_{T \subseteq S} (|S|-(\ell-1))h_T(\Phi, \mathfrak{G}) & - &   \sum_{\substack{S \subseteq [d] \\ |S| = i+1}} \sum_{x \in S} \sum_{T \subseteq S \setminus \{x\}} h_T(\Phi, \mathfrak{G}) &=\\
       & \sum_{\substack{S \subseteq [d] \\ |S| = i+1}} \left(\sum_{T \subseteq S} (|S|-(\ell-1))h_T(\Phi, \mathfrak{G})\right. & - &   \left.\sum_{x \in S} \sum_{T \subseteq S \setminus \{x\}} h_T(\Phi, \mathfrak{G})\right) &=\\
      &   \sum_{\substack{S \subseteq [d] \\ |S| = i+1}} \left( \sum_{T \subseteq S} (|S|-(\ell-1))h_T(\Phi, \mathfrak{G})\right. & - &  \left.\sum_{T \subseteq S } |S \setminus T| h_T(\Phi, \mathfrak{G})\right) &=\\  
        & \sum_{\substack{S \subseteq [d] \\ |S| = i+1}} \sum_{T \subseteq S} (|T| - (\ell-1))h_T(\Phi, \mathfrak{G}) & = & \\ & \sum_{\substack{S \subseteq [d] \\ |S| = i+1}} \sum_{\substack{T \subseteq S \\ |T| \geq \ell}} \binom{|T| - (\ell-1)}{1}h_T(\Phi, \mathfrak{G}) & \geq_{\mathfrak{G}} & 0. 
\end{align*}
The first equality comes from rewriting the $f$-vector in terms of the flag $f$-vector. The second equality comes from rewriting the second summation to replace $(d-|R|)$ with a sum over elements of $[d]\setminus R$. The third equality comes from reindexing the summations by setting $S = R \cup \{x\}$. The fourth equality involves combining the summations, and then expressing the flag $f$-vector in terms of the flag $h$-vector. The fifth equality follows from simplifying the second summations by computing the coefficient of a given $h_T(\Phi, \mathfrak{G})$. Combining like terms, and using the fact that $h_T(\Phi, \mathfrak{G}) = 0$ when $|T| < \ell$, we arrive at the final expression. However, then we can apply Lemma \ref{lem:fundamental2}. Thus $(d-i)f_{i-1}(\Phi, \mathfrak{G}) \leq (i-\ell+2)f_i(\Phi, \mathfrak{G})$ for all $i \geq \ell$.

Applying Proposition \ref{prop:flawlesssequence}, we see that $(f_{\ell-1}(\Phi, \mathfrak{G}), \ldots, f_{d-1}(\Phi, \mathfrak{G}))$ is equivariantly flawless.
\end{proof}

\begin{proof}[Proof of Theorem \ref{thm:intro2}]
We see that \begin{align*}
\sum_{k=\ell}^{j} \binom{d-k}{j-k}\binom{k-(\ell-i)}{i} h_k(\Phi, \mathfrak{G})
&= \sum_{k=\ell}^{j} \binom{d-k}{j-k}\binom{k-(\ell-i)}{i} \sum_{T \subseteq S: |T| =k} h_T(\Phi, \mathfrak{G})\\
&= \sum_{k=\ell}^{j}\sum_{T \subseteq S: |T| =k} \binom{d-k}{j-k}  \binom{|T|-(\ell-i)}{i}h_T(\Phi, \mathfrak{G})\\
&= \sum_{k=\ell}^{j}\sum_{T \subseteq S: |T| =k} \sum_{S \subseteq [d]: |S| = j, T \subseteq S}  \binom{|T|-(\ell-i)}{i}h_T(\Phi, \mathfrak{G})\\
&= \sum_{T \subseteq S: |T| \geq \ell} \sum_{S \subseteq [d]: |S| = j, T \subseteq S}  \binom{|T|-(\ell-i)}{i}h_T(\Phi, \mathfrak{G})\\
& = \sum_{S \subseteq [d]: |S| = j} \sum_{T \subseteq S: |T| \geq \ell} \binom{|T|-(\ell-i)}{i}h_T(\Phi, \mathfrak{G}) \\
& \geq_{\mathfrak{G}} 0
\end{align*}
where the last inequality comes from Lemma \ref{lem:fundamental2}, applied to $S$.
\end{proof}

\section{Mixed Graphs}
\label{sec:mixedgraph}

Given a finite set $V$, a \emph{mixed graph} is a triple $(V, U, D)$, where $U$ is a set of undirected edges, and $D$ is a set of directed edges.  A mixed graph is \emph{acyclic} if it does not contain a directed cycle. A \emph{mixed cycle} in $G$ is a cycle in the underlying undirected graph that has at least one directed edge. A \emph{near cycle} is a mixed cycle of length three that has exactly one undirected edge.

There are two polynomial invariants associated to acyclic mixed graphs: the weak and strong chromatic polynomial, both introduced in \cite{beck-et-al-2}, motivated by work in \cite{beck-et-al}. Given an acyclic mixed graph $\spe{g},$ the \emph{weak chromatic polynomial} $\chi(\spe{g}, k)$ counts the number of functions $f: V \to [k]$ subject to:
\begin{enumerate}
    \item For every $uv \in U$, we have $f(u) \neq f(v)$.
    \item For every $(u,v) \in D,$ we have $f(u) \leq f(v)$.
\end{enumerate}
Let $F(G)$ be the set of all weak graph colorings.
The \emph{strong} chromatic polynomial $\bar{\chi}(\spe{g}, k)$ counts similar functions, only with strict inequalities for the second condition instead of the weak inequality. We introduced quasisymmetric function generalizations of both polynomials in \cite{white-1}, by showing how both polynomial invariants come from characters on $\spe{MG}$.

 An automorphism of a mixed graph is a bijection on the vertices which preserve directed edges and undirected edges. Let $\Aut(G)$ be the group of automorphisms of a mixed graph $G$. 

Let $\mathfrak{G} \subseteq \Aut(G).$ Then $\mathfrak{G}$ acts on $F(G)$, via $\mathfrak{g}f = f \circ \mathfrak{g}^{-1}$ for $f \in F(G)$ and $\mathfrak{g} \in \mathfrak{G}$. Given $\mathfrak{g} \in \mathfrak{G}$, we define 
\[ \chi(G, \mathfrak{G}, \mathbf{x}; \mathfrak{g}) = \sum_{f \in \Fix_{\mathfrak{g}}(F(G))} \mathbf{x}^f. \]
 We show that $\chi(G, \mathfrak{G}, \mathbf{x})$ is a Hilbert quasisymmetric class function for a relative simplicial complex.

Given a mixed graph $G$, we let $P$ be the transitive closure of the relation given by the directed edges. When $G$ has no directed cycles, $P$ is a poset.
Let $I \subseteq J$ be two order ideals. We say that the pair $(I,J)$ is \emph{stable} if there is no undirected edge in $J \setminus I$.
 Let $J(G)$ be the set of order ideals of $P$, and let 
\[\Phi(G) = \{I_1 \subset I_2 \subset \cdots \subset I_k: (I_j, I_{j+1}) \mbox{ is stable for all } 0 \leq j \leq k \}\]
We define $I_0 = \emptyset$, and $I_{k+1} = N$. Then $\Phi(G)$ is a balanced relative simplicial complex, with coloring $\kappa(I) = |I|.$

\begin{proposition}
Let $G$ be a mixed graph with no directed cycles, and let $\mathfrak{G} \subseteq \Aut(G)$. Then $\chi(G, \mathfrak{G}, \mathbf{x}) = \Hilb(\Phi(G), \mathfrak{G}, \mathbf{x})$. Thus $\chi(G, \mathfrak{G}, x) = \ps \Hilb(\Phi(G), \mathfrak{G}, \mathbf{x}).$
\label{prop:graphtocomplex}
\end{proposition}
\begin{proof}
Let $V$ be the vertex set of $G.$
Let $\mathfrak{g} \in \mathfrak{G}$, and write \[\Hilb(\Phi(G), \mathfrak{G}, \mathbf{x}; \mathfrak{g}) = \sum_{I_{\cdot} \in \Phi} \sum_{i_0 < i_1 < \cdots < i_k} x_{i_0}^{|I_1|} x_{i_1}^{|I_2 \setminus I_1|} \cdots x_{i_k}^{|N \setminus I_{k}|}.\] 
Given $I_{\cdot}$ and $(i_0, \ldots, i_k)$, define a function $f:V \to \mathbb{N}$ by $f(n) = i_j$ if $n \in I_{j+1} \setminus I_j,$ where $I_{k+1}=N$. We claim that $f$ is a weak coloring. Let $uv$ be an undirected edge, and suppose $f(u) = i_j$. If $f(v) = i_j$, then $I_{j+1} \setminus I_j$ contains an edge, which is a contradiction. Thus $f(v) \neq i_j$. Let $(u,v)$ be a directed edge, and let $f(u) = i_j$. Then $u \geq v$ in $P$. Since $I_j$ is an ideal, $v \in I_j$, so $f(v) \leq i_j$. Hence $f$ is a weak coloring.

Moreover, every coloring arises from this construction.
Let $f$ be a proper coloring of $G$ that is fixed by $\mathfrak{g}$. Let $I_j = f^{-1}([j])$. Then $I_j \subseteq I_k$ whenever $j \leq k$. If we remove the repeats, we get a sequence $I_1 \subset I_2 \subset \cdots \subset I_k$, which we denote by $I_{\cdot}(f)$. Moreover, we claim each $I_j$ must be an order ideal of $P$. If not, then there exists $x \in I_j$, $y \leq x$ with $y \not\in I_j$. This implies that $f(y) > f(x)$, and that there is a directed path in $G$ from $x$ to $y$. Hence there is an edge $(u,v)$ on the path where $f(u) > f(v)$, a contradiction.

Therefore $I_j$ is an order ideal for all $j$. Similarly, $I_j \setminus I_{j-1}$ is stable for all $j.$ Finally, the chain $I_1 \subset I_2 \subset \cdots \subset I_k$ is fixed by $\mathfrak{g}$. Thus $I_{\cdot}(f) \in \Phi(G)$.
We also write $f(N) = \{i_1, \ldots, i_k \}$ where $i_1 < i_2 < \cdots, i_k$. Then the map $f \mapsto (I_{\cdot}(f), (i_1, \cdots, i_k))$ defines a bijection between colorings and terms of $\Hilb(\Phi(G), \mathfrak{G}, \mathbf{x}; \mathfrak{g}).$
\end{proof}


\begin{proof}[Proof of Theorem \ref{thm:mixedgraphs}]
Let $P_G$ be the partial order given by the transitive closure of the directed edges of $G$. Let $\Delta(P)$ be the order complex of $J(P)$.
We write $\Phi(G) = (\Delta(P), \Gamma(G))$, where $\Gamma(P)$ consists of chains of ideals $I_1 \subset I_2 \subset \cdots \subset I_k$ such that $I_j \setminus I_{j-1}$ contains an undirected edge of $G$ for some $j$. Since $\Delta(J(P))$ is shellable, it is Cohen-Macaulay. If $\Gamma = \emptyset$, then we are done. 

Otherwise, it suffices by Proposition \ref{prop:relative} to show that $\dim \Gamma(G) = \dim \Delta(P) - 1$, and that $\Gamma(G)$ satisfies $(S_{m(G)-1})$. We first show that $\dim \Gamma = \dim(\Delta(P))-1$. To see this, consider any undirected edge $e$ of $G$. Since $G$ has no near-cycles, $P/e$ is still a partial order, so we can consider any linear extension $\ell$ of $P/e$. Then we can view $\ell$ as a set composition $C_1|\cdots|C_k$ of $V$ where $C_i = e$ for some $i$, and the other blocks are singletons. Then $C_1 \subset C_1 \cup C_2 \subset \cdots \subset C_1 \cup \cdots \cup C_k$ is an element of $\Gamma(G)$ of size $\dim(\Delta(P))$.

Suppose that there is only one undirected edge $e$, and that $G$ has no near cycles. Consider the function $f: V(\Gamma(M)) \rightarrow V(J(P/e))$ given by 
\[ f(I) = \begin{cases} I/e & e \subseteq I \\ I & \mbox{ otherwise} \end{cases} \]
Then $f$ is an isomorphism between $\Gamma(G)$ and $\Delta(G/e)$. Hence $\Gamma(G)$ is relatively Cohen-Macaulay.

Suppose that $m(G) = 2$. It suffices to show that $\Gamma(G)$ is non-empty. However, that follows from the fact that it has dimension $\dim(\Delta(P)) - 1.$

So suppose that $m(G) > 2$. Then there are at least two undirected edges. Let $e$ be an undirected edge. Let $U' = U \setminus \{e\}$. Then $\Gamma(G) = \Gamma(G - e) \cup \Gamma(G \setminus U')$. We see that $m(G-e) \geq m(G)$ and $m(G \setminus U') \geq m(G)$. By induction, $\Gamma(G-e)$ and $\Gamma(G \setminus U')$ both satisfy $(S_{m(G)-1})$.

We claim that $\Gamma(G - e) \cap \Gamma(G \setminus U')$ is isomorphic to $\Gamma(G/e)$. 
Let $I \in V(\Gamma(G-e) \cap \Gamma(G \setminus U')$. Then there exists $f \in U'$ such that $I$ is an ideal of $P$ and $e \cup f \subseteq I$. Then $f \subseteq I/e$, so $I/e \in V(\Gamma(G/e))$. If we define $F: V(\Gamma(G-e) \cap \Gamma(G \setminus U')) \rightarrow V(\Gamma(G/e))$ by $F(I) = I/e$, then $F$ is an isomorphism of simplicial complexes. 

We see that $m(G)-1 \leq m(G/e) \leq m(G)$. Moreover, if $G/e$ has a near cycle, then $m(G) = 2.$ Since we assume $m(G) > 2$, it follows that $G/e$ has no near cycle, and hence $\dim(\Gamma(G/e)) = \dim(\Gamma(G)) - 1$. By induction, $\Gamma(G/e)$ satisfies $(S_{m(G)-2})$. By Proposition \ref{prop:constructible}, it follows that $\Gamma(G)$ satisfies $(S_{m(G)-1})$. Thus By Proposition \ref{prop:relative}, $\Phi(G)$ satisfies $(S_{m(G)}).$

Now let $\mathfrak{G} \subseteq \Aut(G)$. By Proposition \ref{prop:graphtocomplex}, we see that $\chi(G, \mathfrak{G}, \mathbf{x}) = \Hilb(\Phi(G), \mathfrak{G}, \mathbf{x})$. Hence $\chi(G, \mathfrak{G}, x) = \ps \Hilb(\Phi(G), \mathfrak{G}, \mathbf{x}).$ We write \[\chi(G, \mathfrak{G}, x) = \sum_{i=0}^n f_{i-2} \binom{x}{i} = \sum_{i=0}^n h_{i-1} \binom{x+n-i}{n}. \]
From principal specialization, it follows that 
\begin{align*} f_{i-1} &= \sum_{S \subseteq [n-1]: |S| = i} [M_{S,n}] \chi(G, \mathfrak{G}, \mathbf{x}) \\
&= \sum_{S \subseteq [n-1]: |S| = i} [M_{S,n}] \Hilb(\Phi(G), \mathfrak{G}, \mathbf{x}) \\ 
&= f_{i-1}(\Phi(G)). \end{align*}
Similarly, $h_i = h_i(\Phi(G)).$
The other results follow from Theorem \ref{thm:intro3} and from Theorem \ref{thm:interpretation}.\end{proof}

\section{Double Posets}
\label{sec:dbl}

Now we will discuss double posets, and marked posets. The Hopf algebra of double posets was introduced by Malvenuto and Reutenauer \cite{malvenuto-reutenauer}. Grinberg associated a quasisymmetric function to any double poset, which is a generalization of Gessel's $P$-partition enumerator. This quasisymmetric function is studied extensively by Grinberg \cite{grinberg}, who proved a combinatorial reciprocity theorem. We studied a quasisymmetric class function associated to a double poset \cite{white-2}.

Given a finite set $N$, a \emph{double poset} on $N$ is a triple $(N, \leq_1, \leq_2)$ where $\leq_1$ and $\leq_2$ are both partial orders on $N$. Often for standard poset terminology, we will use $\leq_i$ as a prefix to specify which of the two partial orders is being referred to. For instance, a $\leq_1$-order ideal is a subset that is an order ideal with respect to the first partial order, and a $\leq_1$-covering relation refers to a pair $(x,y)$ such that $x \prec_1 y$.

\begin{figure}
\begin{center}
\begin{tabular}{cc}
\begin{tikzpicture}
  \node[circle, draw=black, fill=white] (b) at (0,2) {$b$};
  \node[circle, draw=black, fill=white] (a) at (0,0) {$a$};
  \node[circle, draw=black, fill=white] (c) at (2,0) {$c$};
  \node[circle, draw=black, fill=white] (d) at (2,2) {$d$};
 \draw[-Latex] (b) -- (a);
  \draw[-Latex] (d) edge (c);
  \draw[-Latex] (b) -- (c);
    \draw[-Latex] (d) -- (a);

\end{tikzpicture}

& 

\begin{tikzpicture}
  \node[circle, draw=black, fill=white] (b) at (0,2) {$c$};
  \node[circle, draw=black, fill=white] (a) at (0,0) {$b$};
  \node[circle, draw=black, fill=white] (c) at (2,0) {$d$};
  \node[circle, draw=black, fill=white] (d) at (2,2) {$a$};
 \draw[-Latex] (b) -- (a);
  \draw[-Latex] (d) edge (c);

\end{tikzpicture}
\end{tabular}
\end{center}
\caption{A double poset, with $\leq_1$ on the left, and $\leq_2$ on the right.}
\label{fig:doubleposet}
\end{figure}
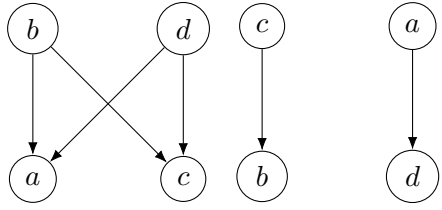

Given a double poset $D$, a pair $(m,m') \in M$ is an \emph{inversion} if $m <_1 m'$ and $m' <_2 m$. 
We refer to an inversion $(m, m')$ as a \emph{descent} if $m \prec_1 m'$. We say that $D$ is \emph{inversion-reducible} if for every inversion $(m, m')$, either $(m,m')$ is a descent, or there exists $m''$ with $m <_1 m'' <_1 m'$ such that $(m, m'')$ or $(m'', m')$ is an inversion. Finally, a double poset is \emph{tertispecial} if, whenever $m \prec_1 m'$, then $m$ and $m'$ are $\leq_2$-comparable.

 Let $D$ be a double poset on a finite set $N$, and let $f: N \rightarrow \mathbb{N}$. Then $f$ is a $D$-partition if and only if it satisfies the following two properties:
\begin{enumerate}
    \item For $i \leq_1 j$ in $D$, we have $f(i) \leq f(j)$.
    \item For $i \leq_1 j$ and $j \leq_2 i$ in $D$, we have $f(i) < f(j)$.
\end{enumerate}
We let $P_D$ be the set of $D$-partitions.

Given a double poset $D$, and a permutation $\mathfrak{g} \in \mathfrak{S}_N$, we say $\mathfrak{g}$ is an \emph{automorphism} of $D$ if for all $i, j \in N$, and $k \in \{1,2 \}$, if $i \leq_k j$, then $\mathfrak{g}(i) \leq_k \mathfrak{g}(j)$. We let $\Aut(D)$ be the automorphism group of $D$. For instance, for the double poset in Figure \ref{fig:doubleposet}, the permutation $(ac)(bd)$ is the only nontrivial automorphism. Similarly the only nontrivial automorphism of the double poset in Figure \ref{fig:specialposet} is the permutation $(a)(bd)(c)$.

Let $\mathfrak{G} \subseteq \Aut(D)$. For $\mathfrak{g} \in \mathfrak{G}$ and $f: N \rightarrow \mathbb{N}$, let $\mathfrak{g} \cdot f$ be defined by $(\mathfrak{g} \cdot f)(v) = f(\mathfrak{g}^{-1} \cdot v)$ for all $v \in N$. This defines an action of $\mathfrak{G}$ on $P_D$. Moreover, we see that $\mathbf{x}^f = \mathbf{x}^{\mathfrak{g} \cdot f}$.

For a double poset $D$ on $N$, $\mathfrak{G} \subseteq \Aut(D)$, and $\mathfrak{g} \in \mathfrak{G}$, let \[\Omega(D, \mathfrak{G}, \mathbf{x}; \mathfrak{g}) = \sum_{f \in P_D: \mathfrak{g}f = f} \mathbf{x}^f.\]
We call $\Omega(D, \mathfrak{G}, \mathbf{x})$ the \emph{$D$-partition quasisymmetric class function}. We introduced this invariant in \cite{white-2}. 

As an example, consider the double poset $D$ in Figure \ref{fig:doubleposet}, and let $\mathfrak{G} = \Aut(D)$. Let $\rho$ denote the regular representation. Then \begin{align*} \Omega(D, \mathfrak{G}, \mathbf{x}) & = M_{\{2\},4}+\rho(M_{\{2,3\},4}+M_{\{1,3\},4}+M_{\{1,2\},4}+2M_{\{1,2,3\},4}) \\ 
& = F_{\{2\},4}+\sgn (F_{\{2,3\},4}+F_{\{1,2\},4} - F_{\{1,2,3\},4}) + \rho F_{\{1,3\},4}. 
\end{align*}

As another example, consider the double poset $D$ in Figure \ref{fig:specialposet}, and let $\mathfrak{G} = \Aut(D)$. Let $\rho$ denote the regular representation. Then \begin{align*} \Omega(D, \mathfrak{G}, \mathbf{x}) & = M_{\{1\},4}+M_{\{1,3\},4}+\rho(M_{\{1,2\},4}+M_{\{1,2,3\},4}) \\ 
& = F_{\{1\},4}+\sgn F_{\{1,2\},4}.
\end{align*}

\begin{figure}
\begin{center}
\begin{tabular}{cc}
\begin{tikzpicture}
\node[circle,draw=black] (A) at (0,1) {a};
\node[circle,draw=black ] (B) at (1,0) {b};
\node[circle,draw=black] (C) at (0,-1) {c};
\node[circle,draw=black] (D) at (-1,0) {d};

\draw[-Latex] (A) -- (B);
\draw[-Latex] (B) -- (C);
\draw[-Latex] (D) -- (C);
\draw[-Latex] (A) -- (D);
\end{tikzpicture}
&
\begin{tikzpicture}
  \node[circle, draw=black, fill=white] (b) at (0,2) {$a$};
  \node[circle, draw=black, fill=white] (a) at (0,0) {$b$};
  \node[circle, draw=black, fill=white] (c) at (2,0) {$d$};
  \node[circle, draw=black, fill=white] (d) at (2,2) {$c$};
 \draw[-Latex] (b) -- (a);
  \draw[-Latex] (d) edge (c);
  \draw[-Latex] (b) -- (c);
    \draw[-Latex] (d) -- (a);

\end{tikzpicture}
\end{tabular}
\end{center}
\caption{A double poset, with $\leq_1$ on the left, and $\leq_2$ on the right.}
\label{fig:specialposet}
\end{figure}
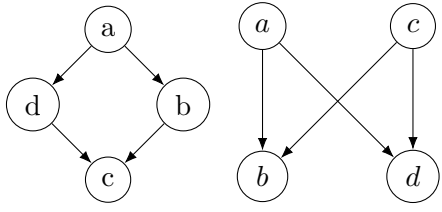

\subsection{From Double Posets to Mixed graphs}
Let $D$ be a double poset on a finite set $N$.
We can associate a mixed graph $G(D)$ to an inversion-reducible double poset $D$. We let $D_1 = \{(u,v): v \mbox{ $\leq_1$-covers } u \}.$ 
We let $U = \{uv:  v \mbox{ $\leq_1$-covers } u, v \leq_2 v \}.$ Then $G(D) = (V(P), U, D_1).$ We see that, if $\mathfrak{g}$ is an automorphism of $D$, then it is also an automorphism of $G(D).$ Thus, $\Aut(D) \subseteq \Aut(G(D)).$

Moreover, we see that a function $f:N \to \mathbb{N}$ is a $D$-partition if and only if $f$ is a weak coloring of $G(D).$ 
Thus $\Omega(D, \mathfrak{G}, \mathbf{x}) = \chi(G(D), \mathfrak{G}, \mathbf{x}).$

Thus, we can transfer the results from the previous section to derive results about double posets. 
First, $G(D)$ has no near cycles. If $G(D)$ had a near cycle $C$ with undirected edge $uv$, and there is a directed path from $u$ to $v$. By definition of $G(D)$, since $uv \in E(G(P))$, we know $v$ $\leq_1$-covers $u$ and $v \leq_2 u$. However, there is also a directed path of length at least two from $v$ to $u$, which implies that $u \leq_1 v$ but that $v$ does not cover $u$. Hence there is no near cycle. From Theorem \ref{thm:mixedgraphs} we conclude the following:
\begin{theorem}
Let $D$ be a double poset on $n$ elements. Suppose that $D$ is inversion-reducible, and let $m(D) = m(G(D)).$ Let $\mathfrak{G}$ act on $D$. Then $\mathfrak{G}$ acts on $G(D)$, and $\Omega(D, \mathfrak{G}, \mathbf{x}) = \chi(G(D), \mathfrak{G}, \mathbf{x}).$

For all $S \subseteq [n-1]$, if $|S| \leq m(D)$, we have $[F_{S,n}]\Omega(D,\mathfrak{G}, \mathbf{x}) \geq 0.$

\end{theorem}

Now we apply our results to tertispecial posets.
\begin{proposition}
Suppose that $D$ is a tertispecial poset on $n$ elements. Then $D$ is inversion reducible. Moreover $G(D)$ has no mixed cycles. Thus, for every $\mathfrak{G} \subseteq \Aut(D)$ and every $S \subseteq [n-1]$, we have $[F_{S, n}]\Omega(D, \mathfrak{G}, \mathbf{x}) \geq 0.$
\end{proposition}
\begin{proof}
We prove the first condition by induction on $n$. Suppose that $(x,y)$ is an inversion pair. if $n = 2$, then the result is immediate. So suppose $n > 2.$

Choose $t$ such that
 $x \prec_1 t \leq_1 y$. Since $D$ is tertispecial, we have $x \leq_2 t$ or $t \leq_2 x$. In the latter case, we have found a descent pair $(x,t)$. In the former case the pair $(t,y)$ forms an inversion pair. We observe that the interval $[t,y]$ is equal to $(y) \setminus (t)$, where $(a)$ is the principal order ideal generated by $a$. Since $|[t,y]| < |N|$, by induction there is a descent pair $(w,z)$ in $[t,y]$, and hence $[x,y]$ also has a descent.
 
 To see that $G(D)$ has no mixed cycles, let $(u,v)$ be a directed edge with $v \prec_1 u$. Then we have $u \not\leq_2 v$. However, by the definition of tertispecial, $u$ and $v$ are $\leq_2$-comparable, so we have $v \leq_2 u$. We see then that for every directed edge $(u,v)$, we have $v \leq_2 u$. Hence $D(P)$ must be acyclic. The rest follows from the previous Theorem.
\end{proof}
Thus we have shown Theorem \ref{thm:introposet2}.

\newcommand{\etalchar}[1]{$^{#1}$}
\providecommand{\bysame}{\leavevmode\hbox to3em{\hrulefill}\thinspace}
\providecommand{\MR}{\relax\ifhmode\unskip\space\fi MR }
\providecommand{\MRhref}[2]{%
  \href{http://www.ams.org/mathscinet-getitem?mr=#1}{#2}
}
\providecommand{\href}[2]{#2}

\end{document}